\documentclass[a4paper]{article}
\usepackage{geometry}
\usepackage[utf8]{inputenc}
\usepackage[T1]{fontenc}
\usepackage{longtable}
\usepackage{lmodern, setspace}
\usepackage{amsthm, amsmath, amssymb, mathtools, mathrsfs, stmaryrd, bbm}
\usepackage{tikz}
\usepackage{tikz-cd}
\usetikzlibrary{calc,babel}
\usepackage{accents}
\usepackage{scalerel}
\usepackage{todonotes}
\usepackage{pdflscape}
\usepackage{adjustbox}
\usepackage{MA_Titlepage}
\usepackage{graphicx}
\graphicspath{ {./images/} }
\usetikzlibrary{arrows}
\usetikzlibrary{babel}
\usepackage{quiver}  
\usepackage{macros}
\usepackage{rotating}
\usepackage{ stmaryrd }

\usepackage[unicode, colorlinks=true, linkcolor={blue}]{hyperref}
\usepackage[nameinlink, capitalise, compress]{cleveref}
\crefname{diag}{Diagram}{Diagrams}
\creflabelformat{diag}{(#2\textup{#1})#3}
\hypersetup{colorlinks,linkcolor={blue},citecolor={green},urlcolor={red}} 

\theoremstyle{definition}
\newtheorem{thm}{Theorem}[section]

\newtheorem{prop}[thm]{Proposition}
\newtheorem*{prop*}{Proposition}
\newtheorem{lem}[thm]{Lemma}
\newtheorem{ex}[thm]{Example}
\newtheorem{defn}[thm]{Definition}

\newtheorem{rmk}[thm]{Remark}

\newtheorem*{con*}{Construction}
\newtheorem{notation}[thm]{Notation}
\newtheorem*{not*}{Notation}
\newtheorem*{obs*}{Observation}

\DeclareMathOperator{\Spec}{Spec}

\makeatletter
\renewcommand\subsubsection{\@startsection{subsubsection}{3}{\z@}%
                                     {-3.25ex\@plus -1ex \@minus .2ex}%
                                     {-1em}%
                                     {\normalfont\normalsize\bfseries}}
\makeatother

\authornew{Paul Immanuel}
\geburtsdatum{8th July 1999,}
\geburtsort{Choondal, Kerala.}
\date{26th February 2024}

\betreuer{Advisor: Dr. Arnaud Eteve}
\zweitgutachter{Second Advisor: Prof. Dr. Jessica Fintzen}

\institut{Max-Planck-Institut für Mathematik} 
\institutb{Mathematisches Institut} 
\title{Incarnations of the Fourier Transform in Algebraic Geometry}
\ausarbeitungstyp{Master's Thesis  Mathematics}

\begin{document}

\maketitle

\tableofcontents

\section{Introduction}
The \textit{“$\ell$-adic Fourier transform”} was first introduced by Deligne in a letter to Kazhdan in 1976. Subsequently, Laumon \cite{laumon1987transformation} used this Fourier transform to give a simplified proof of the Weil conjectures. The main ingredients used in the definition, were the construction of the derived category of $\overline{\mathbb{Q}}_\ell$-adic sheaves, and that of the derived pushforward with compact supports. Let  $\mathbb{A}_0$ denote the affine line over a finite field $k$ of characteristic $p$. Let $\pi_1,\pi_2:\mathbb{A}_0\times\mathbb{A}_0\rightarrow\mathbb{A}_0$ correspond to the two projections and $m:\mathbb{A}_0\times\mathbb{A}_0\rightarrow\mathbb{A}_0$ correspond to the multiplication map given by $m(x,y)=xy$. The Fourier transform is then defined as the functor
$$T_\psi: D^b_c(\mathbb{A}_0,\overline{\mathbb{Q}}_l) \rightarrow D^b_c(\mathbb{A}_0,\overline{\mathbb{Q}}_l)$$ given by
$$T_\psi(K_0)=R {\pi_1}_!({\pi_2}^*(K_0)\otimes m^*\mathcal{L}_0(\psi))[1]$$
where $\mathcal{L}_0(\psi)$ is the rank one $\overline{\mathbb{Q}}_\ell$-local system on $\mathbb{A}_0$
induced by $\psi$. Through Grothendieck's framework of the function sheaf dictionary, we see that up to a sign, the $\ell$-adic Fourier transform agrees with the function theoretic transform over the finite group $\mathbb{A}(k)$. Moreover, we see that the transform behaves very similarly to the classical case, with the inversion theorem holding up to a Tate twist. \\

An extension of these ideas to the case of perfect unipotent group schemes is heavily used \cite{boyarchenko2013charactersheavesunipotentgroups} and other modern treatments to gain representation theoretic data about these objects. The correct notion of duality in this setting, inspired by Serre, and later expanded on by Begueri in \cite{begueri1980dualite}, was a key step in this generalization. Explicitly, let $\Gamma$ be a connected commutative unipotent algebraic group, then the functor $\Gamma^*$ (called the Serre dual), on the category $\mathfrak{Perf}_k$ of perfect schemes of characteristic $p$ over a perfect field $k$ given by, 
$$\Gamma^*(S)=\mathrm{Ext}^1_S(\Gamma_S, \mathbb{Q_p}/\mathbb{Z}_p)$$
is exact and representable by a connected commutative algebraic unipotent group. Moreover, if $\Gamma$ is also perfect, the functor $\Gamma \rightarrow \Gamma^*$ is an anti-auto-equivalence on the category $\mathfrak{cpu}_k^\circ$ of connected commutative perfect unipotent group schemes over $k$.\\

In fact, a slight modification of this construction allows us to generalize this to the case of non-commutative unipotent group schemes. This is done by restricting the previous functor to only central extensions. Once again one can show that this functor is representable, however it fails to be involutive as before.\\

In the perfectoid world, the analogous objects that considered are Banach-Colmez spaces associated to vector bundles on the Fargues-Fontaine curve $X_S$ with only positive or only negative slopes. For these objects we are able to define a Fourier transform, 
$$\mathcal{F}_\psi:D_{ét}(BC(\mathcal{E}),\Lambda)\rightarrow D_{ét}(BC(\mathcal{E^\vee}),\Lambda)$$
which gives an equivalence of categories. A key aspect in the definition, in complete analogy the affine line, is the construction of an element $\mathcal{L}_\psi\in D_{ét}([S/\underline{G}],\Lambda)$. However, unlike in the affine case, the theory of duality is quite a bit more complicated, and this is a primary focus of \cite{anslebras}.

\subsection{Acknowledgements}

To my advisor, Arnaud Eteve, I am truly thankful for the time and effort you have spent on me. The respect and sincerity you showed towards me went a long way in helping me feel more assured and comfortable during this past year. \\

To the mathematical community in Bonn, to whom I owe a great deal for lifting me to a level higher than was ever possible alone. Especially Ferdinand Wagner, Chenji Fu, Thomas Manopulo, Konrad Zou, Michele Lorenzi. To Dr. Gebhard Martin, my lecturer in algebraic geometry for the clarity and intricacies in your lectures.\\

To my friends Chen, Pádraig, Ramon and Louis for your companionship, warmth and laughs, and for making Bonn home. To my coaches and captains, Stefan Reddemann and Salah Mirza, for the way you invited me into your teams and gave me a place to play.\\

To my family for the deepest love, to Kyra for everything.

\section{Preliminaries}
In this section, we set up some fundamental concepts that help us build the theory of the Fourier transform on the affine line. In particular, we discuss the  construction of the étale fundamental group and introduce the derived category of $\overline{\mathbb{Q}}_l$-adic sheaves. We then go on to use these tools to explain the framework of the function-sheaf dictionary, and state the Grothendieck trace formula. These turn out to be   critical tools in relating the Fourier transform on the affine line to the more classical setting. The main reference will be \cite{kiehl2013weil}.

\begin{subsection}{The Étale fundamental group}
    Let $X$ be a connected scheme. We say that $f:Y\rightarrow X$ is an étale cover if $f$ is a finite étale morphism. Let $\mathbf{FEt}/X$ be the full subcategory of $\mathbf{Sch}/X$ schemes over $X$ generated by such étale covers. The fact that étale morphisms are flat, and the connectivity assumption on the base, allows us to have a well defined notion of the \textbf{degree} of a cover.\\
    
    By a geometric point $\overline{x}$ of $X$, we mean a morphism $\Spec(L) \rightarrow X$ mapping to a point $x$ of $X$, where $L$ is a separable closure of the residue field at $x$. We call $(X^\prime,\overline{x}^\prime)$ a pointed étale cover of $X$ if, $X^\prime$ is an étale cover and $\overline{x}^\prime$ is a geometric point that lifts $\overline{x}$. We now introduce two key lemmas which give us insight into the structure of $\mathbf{FEt}/X$.

    \begin{lem}\cite[Lemma~1.2.11]{Brianconradet}(Rigidity of pointed étale covers)\label{2.1}
         If $f, g: X^{\prime} \rightrightarrows X^{\prime \prime}$ are $X$-maps to a separated étale $X$-scheme $X^{\prime \prime}$, and $\left(X^{\prime}, \bar{x}^\prime\right)$ is a pointed connected scheme such that $f\left(\bar{x}^{\prime}\right)=g\left(\bar{x}^{\prime}\right)$ in $X^{\prime \prime}(L)$, then $f=g$.

\begin{proof}The closed immersion
$$
\Delta: X^{\prime \prime} \rightarrow X^{\prime \prime} \times_X X^{\prime \prime}
$$
is étale, hence open, so $X^{\prime \prime} \times_X X^{\prime \prime}$ can be written as $X^{\prime \prime} \amalg Y$ with $X^{\prime \prime}$ equal to the diagonal. Since $X^{\prime}$ is connected and (by assumption) the image of the map $f \times g: X^{\prime} \rightarrow X^{\prime \prime} \times_X X^{\prime \prime}$ meets $\Delta$, the image lies in $\Delta$.

\end{proof}
\end{lem}
 This lemma tells us that the order of the group of automorphisms of an étale cover; $|\operatorname{Aut_X}Y|$ is always less than the degree $n$ of the cover. This is because the fiber over the geometric point $\overline{x}$ has exactly  $n$ points.

    \begin{defn}
      An étale cover $Y$ is said to be \textit{Galois} if it is connected and  $|\operatorname{Aut_X}Y|=\mathbf{degree}(Y)$.
    \end{defn}
\begin{lem}\label{2.3}
For every connected étale cover $X^\prime\rightarrow X$, we can find a Galois cover $X^{\prime\prime}\rightarrow X$ which factorizes as a surjection $X^{\prime\prime}\rightarrow X^{\prime}$ . Furthermore, for any two Galois covers $Z$ and $Y$, we can find a third $\widetilde{X}$, with surjective morphisms to both.
\end{lem}

The rigidity lemma tells us that there is at most one map between pointed étale covers $(X^\prime,\overline{x}^\prime)$ and $(X^{\prime\prime},\overline{x}^{\prime\prime})$ sending $\overline{x}^\prime$ to $\overline{x}^{\prime\prime}$. Furthermore, it tells us that we have a well defined surjective homomorphism of groups  $\operatorname{Aut_X}X^\prime \rightarrow \operatorname{Aut_X}X^{\prime\prime}$. In addition Lemma \ref{2.3} tells us that these groups form an inverse system, leading us to the following definition, 

\begin{defn}
    The \textit{étale fundamental group} of a connected scheme $X$ (with respect to a geometric point $\overline{x}$ ) is the profinite group
$$\pi_1(X,\overline{x})=\varprojlim_{(X^\prime,\overline{x}^\prime)}\operatorname{Aut}_XX^\prime$$
where the inverse limit is taken over connected Galois covers, equipped with a geometric point lying above $\overline{x}$.
\end{defn}

\begin{ex}\label{2.5}
Let us look at the special case where $X=\Spec(K)$ is a point and $K$ is a field. Let us also fix a algebraic closure $\overline{K}$ of $K$ that acts as a geometric point. The finite étale covers of covers are precisely the finite separable field extensions $K^\prime$ of $K$. The Galois covers can be seen to be precisely the Galois extensions of $K$, i.e the separable normal finite extensions.  The group $\operatorname{Aut_K}K^\prime$ is therefore the Galois group of $K^\prime$, and the inverse limit along all the finite Galois extensions is precisely the absolute Galois group, $\operatorname{Gal}(\overline{K}/K)$. 
\end{ex}
\begin{defn}\cite[Definition~2.3]{Jtsim}
    A universal cover $\widetilde{X}$ of X consists of the following data,
    \begin{itemize}
        \item  A partially ordered set I, which is filtered in the sense that for any two objects, there is some object less than or equal to both of them.
        \item For each $i \in I$ a Galois cover $X_i$ of $X$.
        \item  For any two objects $i, j \in I$ with $i<j$ a transition morphism $\phi_{i j}$ : $X_i \rightarrow X_j$ such that $\phi_{j k} \circ \phi_{i j}=\phi_{i k}$.
    \end{itemize}
    
Such that any connected $X' \in \mathbf{FEt}/X$ is covered by $X_i$ for some $i \in I$. Moreover, we say that $\tilde{X}$ is based, and we write $\tilde{\bar{x}}$ for its base point if each $X_i$ is assigned a geometric point $\overline{x}_i$ in a compatible way with the transition maps.
\end{defn}

Lemmas \ref{2.1} and \ref{2.3} give us the existence of a universal cover $(\widetilde{X},\widetilde{\overline{x}})$. For any scheme $Z$ we define the set $\operatorname{Hom}_X(\widetilde{X},Z)$ to be $\varprojlim_i\operatorname{Hom}_X(X_i,Z)$.
\begin{lem}\label{2.7}
    We fix geometric point $\overline{x}$ of $X$. For any $X'\in\mathbf{FEt}/X$, we have a natural isomorphism $\theta_{X'}:\operatorname{Hom}_X(\widetilde{X},X')\rightarrow X'_{\overline{x}}$ prescribed by $\phi\mapsto\phi(\widetilde{\overline{x}})$.
\end{lem}

Consider now the left action of $\pi_1(X,\overline{x})$ on $\operatorname{Hom}_X(\widetilde{X},X')$ given by defining the action of $g\in \operatorname{Aut}_XX_i$ on $\phi\in\operatorname{Hom}_X(X_i,X')$; $g \cdot\phi=\phi\circ g^{-1}$. We call the $\pi_1(X,\overline{x})$-set obtained from $X'$ in this way, $F_{X'}$.

\begin{prop}\cite{zbMATH03370468}
    There is an equivalence of categories between $\mathbf{FEt}/X$ and the category of finite discrete sets with a continuous $\pi_1(X,\overline{x})$-action, given by the functor $X'\mapsto F_{X'}$. This functor is often called the fiber functor.
\end{prop}

Now we move on to consider another presentation of these categories using the language of sheaves. We denote by $Ét(X)$ the étale topos of a scheme $X$. For a set $\Sigma$, we denote by $\underline{\Sigma}_X$ the sheafification of the presheaf $U\mapsto\Sigma$. We call any $\mathscr{F}\inÉt(X)$  a constant sheaf if it is isomorphic to $\underline{\Sigma}_X$ for any $\Sigma$.
\begin{defn}
   An object $\mathscr{F}\in Ét(X)$ is called locally constant if there exists a covering $(X_i\rightarrow X)$ in the étale site of $X$ such that each $\mathscr{F}|_{X_i}$ is constant. If moreover, the associated set over each $X_i$ is finite, we call $\mathscr{F}$ locally constant constructible. We often abbreviate this to "LCC" for short.
\end{defn}
\begin{ex}\label{2.10}
    Let $f:Y\rightarrow X$ be a morphism of schemes. Denote by $\underline{X}'$ the constant sheaf associated to  $X'\in \mathbf{FEt}/X$. We have that  $f^*(\underline{X}')=\underline{Y\times_X X'}$. Indeed this follows from,

    \begin{align*}
        \operatorname{Hom_Y}(f^*\underline{X}',\mathscr{F})&\cong\operatorname{Hom_Y}(\underline{X}',f_*\mathscr{F})\\
        &\cong f_*\mathscr{F}(X)\\
        &\cong \mathscr{F}(Y\times_XX')\\
        &\cong\operatorname{Hom_Y}(\underline{Y\times_XX'},\mathscr{F})
    \end{align*}

\end{ex}

\begin{thm}\label{2.11}[Classification of LCC sheaves] The functor from $\mathbf{FEt}/X$ to $Ét(X)$ given by $X'\mapsto\underline{X}'$, restricts to an equivalence of categories from finite étale schemes to LCC sheaves. 
    
\end{thm}

We now construct a functor that relates LCC sheaves, to finite discrete $\pi_1(X,\overline{x})$-sets which is also called the fiber functor. Given an LCC sheaf $\mathscr{F}$ over $(X,\overline{x})$ the pull back sheaf $\overline{x}^*\mathscr{F}$ is, by example \ref{2.10}, a constant sheaf associated to a finite étale scheme over $\overline{x}$. This can be completely described using its global section on $\overline{x}$, which again by example \ref{2.10} is precisely the fiber over the geometric point $\overline{x}$ which inherits the same action as prescribed by lemma \ref{2.7}. Let us make this a bit more precise by introducing a more formal definition of the stalk functor.

\begin{defn}\label{2.12}
    Let $\overline{x}:\operatorname{Spec}k\rightarrow X$ be a geometric point of $X$. An étale neighborhood of $\overline{x}$ is a pair $(U,\overline{u})$, where $U\rightarrow X$ is étale and the diagram,
\[\begin{tikzcd}
	{\operatorname{Spec}k} & U \\
	& X
	\arrow["{\overline{u}}", from=1-1, to=1-2]
	\arrow["{\overline{x}}"', from=1-1, to=2-2]
	\arrow[from=1-2, to=2-2]
\end{tikzcd}\]
commutes. A morphism of étale neighborhoods $(U,\overline{u})\rightarrow (V,\overline{v})$ is an $X$ morphism $U\rightarrow V$ that makes $(U,\overline{u})$ into an étale neighborhood of the geometric point $\overline{v}$ of $V$. Finally, if $\mathcal{F}$ is an étale sheaf on $X$, one sets,
$$\mathcal{F}_{\overline{x}}=\underset{(U,\overline{u})}{\mathrm{colim}} \ \Gamma(U,\mathcal{F})$$
where the colimit is taken over all the étale neighborhoods of $\overline{x}$. This is called the stalk of $\mathcal{F}$ at $\overline{x}$.
\end{defn}

\begin{rmk}\label{2.13}
     If $\sigma\in\mathrm{Gal}(k(\overline{x})/k(x))$, and $(U,\overline{u})$ is an étale neighborhood of $\overline{x}$, we have that $(U,\overline{u}\circ\operatorname{Spec}(\sigma))$ is also an étale neighborhood. We thus get an action on the stalk of an étale sheaf on $X$ by sending elements indexed by the pair $(U,\overline{u})$ to the index $(U,\overline{u}\circ\operatorname{Spec}(\sigma))$. For sheaves over spectrums of fields, this action coincides with the action prescribed by lemma \ref{2.7}, considering example \ref{2.5}. 
\end{rmk}

The previous discussions allow us to relate the various categories at play using the following diagram,

\[\begin{tikzcd}
	{\mathbf{FEt}/X} && {\mathbf{LCC}/X} \\
	{} \\
	& {\text{finite discrete}\ \pi_1(X,\overline{x})-\text{sets}}
	\arrow["{X'\mapsto F_{X'}}"', squiggly, tail reversed, from=1-1, to=3-2]
	\arrow["{X'\mapsto\underline{X}'}"', squiggly, tail reversed, from=1-3, to=1-1]
	\arrow["{\mathscr{F}\mapsto \mathscr{F}_{\overline{x}}}", squiggly, tail reversed, from=1-3, to=3-2]
\end{tikzcd}\]

\end{subsection}
\begin{subsection}{$\overline{\mathbb{Q}}_l$-sheaves}
  We first introduce the category of $\overline{\mathbb{Q}}_l$-sheaves. This will be important in our attempt to relate the Fourier transform on the affine line to the more classical Fourier transform for finite groups. This section closely follows \cite{vcesnavivciusadic}.

\begin{notation}
    Let E denote a finite extension of $\mathbb{Q}_l$. We will denote by $\mathfrak{o}$ its ring of integers, with uniformizer $\pi\in \mathfrak{o}$. We will write $\mathfrak{o}_i$ for $\mathfrak{o}/\pi^i\mathfrak{o}$. Sometimes, we work with a tower of extensions $F'/F/E$ for which we have $\mathfrak{O}',\mathfrak{O},\mathfrak{o}$ the corresponding ring of integers, as well as $(\Pi',\Pi,\pi)$ their respective uniformizers. In the following we denote by $X$, a finite type separated scheme over the finite field $\mathbb{F}_q$ of characteristic $p$.
\end{notation}

\begin{defn}
   A sheaf $\mathscr{F}$ on $X$ is called called constructible if $X$ can be written as the finite union of locally closed subsets $X_i$ such that $\mathscr{F}|_{X_i}$ is LCC.

\end{defn}
\begin{subsubsection}{Artin-Rees categories} A \textit{projective system} of sheaves, is a sequence of $\mathfrak{o}$-modules $\{\mathscr{F}_i\}_{i\geq1}$ with structure morphisms $\mathscr{F}_{i+1}\rightarrow\mathscr{F}_{i}$. Morphisms of projective systems are given by level wise morphism that commute with the structure morphisms.\\

One can form shifts of a projective system $\mathscr{F}[n]$ by setting $(\mathscr{F}[n])_i=\mathscr{F}_{i+n}$, with the obvious structure morphisms. Furthermore, the shifts come equipped with a morphism $\mathscr{F}[n]\rightarrow\mathscr{F}$, given by iterating the structure morphisms. We call a projective system a \textit{null system} if there exists some $n\geq0$, such that $\mathscr{F}[n]\rightarrow\mathscr{F}$ is the zero morphism.\\

 The full subcategory of Null systems $\mathcal{N}$, form a \textit{Serre} subcategory of the abelian category of projective systems, $\mathcal{P}$. We refer to the quotient by the Serre subcategory of null systems, as the Artin-Rees category. The homomorphisms between two objects $\mathscr{F}$ and $\mathscr{G}$ in $\mathcal{P}/\mathcal{N}$, are given by $\operatorname{Hom}_{\mathcal{P}/\mathcal{N}}(\mathscr{F},\mathscr{G})=\varinjlim\operatorname{Hom}_\mathcal{P}(\mathscr{F}',\mathscr{G/G'})$, where $\mathscr{F}'$ is a subobject of $\mathscr{F}$, such that $\mathscr{F}/\mathscr{F}'\in\mathcal{N}$ and $\mathscr{G}'$ is a subobject of $\mathscr{G}$ such that $\mathscr{G}'\in \mathcal{N}$.

\end{subsubsection}

\begin{subsubsection}{$\pi$-adic sheaves}
    We define a $\pi$-adic sheaf to be a projective system $\mathscr{F}_{i\geq1}$ where $\pi^i\mathscr{F}_i=0$ such that the structure morphisms are induced from $\mathscr{F}_{i+1}/\pi^i\mathscr{F}_{i+1}\xrightarrow{\cong}\mathscr{F}_i$. The category of $\pi$-adic sheaves however, is not abelian. We rectify this issue by declaring that an Artin-Rees (A-R) $\pi$-adic sheaf is a sheaf in the A-R category that is isomorphic to a $\pi$-adic sheaf. We call the full subcategory of (A-R) $\pi$-adic sheaves $\operatorname{Sh_{A-R}}(X,\mathfrak{o})$, which is an abelian category. If we denote by $\operatorname{Sh}(X,\mathfrak{o})$ the full subcategory spanned by $\pi$-adic sheaves, then we have the following,
    \begin{prop}
        The natural functor from the category of $\pi$-adic sheaves to $\operatorname{Sh}(X,\mathfrak{o})$ is an isomorphism of categories.
    \end{prop}

Since by definition the category $\operatorname{Sh_{A-R}} (X,\mathfrak{o})$ is equivalent to its full subcategory $\operatorname{Sh}(X,\mathfrak{o})$, we can construct a fully faithful functor $E_\mathfrak{o}:\operatorname{Sh_{A-R}} (X,\mathfrak{o})\rightarrow\operatorname{Sh}(X,\mathfrak{o})$, which allows us to view the homomorphism groups in $\operatorname{Sh_{A-R}}(X,\mathfrak{o})$ as $\mathfrak{o}$-modules. This is because $\operatorname{Hom}_{\operatorname{Sh}(X,\mathfrak{o})}(\mathscr{F},\mathscr{G})=\varprojlim\operatorname{Hom}(\mathscr{F}_i,\mathscr{G}_i)$, and each $\operatorname{Hom}(\mathscr{F}_i,\mathscr{G}_i)$ is an $\mathfrak{o}_i$-module.\\

Let $e$ be a natural number. Notice now that the categories of $\pi$-adic sheaves, and that of $\pi^e$-adic sheaves are equivalent. Given a $\pi$-adic sheaf $\mathscr{F}$, one can construct a $\pi^e$-adic sheaf $\mathscr{G}$, by setting $(\mathscr{G})_i=\mathscr{F}_{ie}$, where the transition morphisms $\mathscr{G}_{i+1}/\pi^{ie}\mathscr{G}_{i+1}\xrightarrow{\cong}\mathscr{G}_i$ are induced from the ones of $\mathscr{F}$.\\

Recall that if $F/E$ is a finite extension, for which $\Pi\in F$ and $\pi\in E$ are the respective uniformizers, we have that $\pi\subset(\Pi^e)$ for some natural number $e$ which is the ramification index. This means, that for a $\pi$-adic sheaf $\mathscr{F}$, we can associate a $\Pi^e$-adic sheaf $\mathscr{G}=\{\mathscr{F}_i\otimes_{\mathfrak{o}_i}\mathfrak{O}_{ie}\}_{i\geq0}$, and thus a $\Pi$-adic sheaf which we shall call $\widetilde{S}_{\mathfrak{O}/\mathfrak{o}}(\mathscr{F})$. The composition of the functors $E_\mathfrak{o}$ and $\widetilde{S}_{\mathfrak{O}/\mathfrak{o}}$ gives us functors ${S}_{\mathfrak{O}/\mathfrak{o}}:\operatorname{Sh_{A-R}} (X,\mathfrak{o})\rightarrow \operatorname{Sh_{A-R}} (X,\mathfrak{O})$ which satisfy the following natural compatibilities; ${S}_{\mathfrak{o}/\mathfrak{o}}\cong \operatorname{Id},{S}_{\mathfrak{O}'/\mathfrak{O}}\circ{S}_{\mathfrak{O}/\mathfrak{o}}\cong{S}_{\mathfrak{O}'/\mathfrak{o}}$.

\end{subsubsection}

\begin{subsubsection}{$\overline{\mathbb{Q}_l}$-sheaves.}\label{2.2.3} We first discuss the construction of the category $\operatorname{Sh}(X,E)$. The objects of this category are those of $\operatorname{Sh_{A-R}}(X,\mathfrak{o})$. For $\mathscr{F}\in\operatorname{Sh_{A-R}}(X,\mathfrak{o})$, we write $\mathscr{F}\otimes E$ when we view it as an object in $\operatorname{Sh}(X,E)$. Recall by the previous discussion that homomorphism group between two objects $\mathscr{F},\mathscr{G}\in\operatorname{Sh_{A-R}}(X,\mathfrak{o})$, has an $\mathfrak{o}$-module structure. We define the homomorphism set between $\mathscr{F}\otimes E,\mathscr{G}\otimes E\in\operatorname{Sh}(X,E)$ to be $\operatorname{Hom}(\mathscr{F},\mathscr{G})\otimes_\mathfrak{o}E$. The composition of morphism is given by composition in the first factor and multiplication in the second factor. For an extension of fields $F/E$, the previously constructed functors ${S}_{\mathfrak{O}/\mathfrak{o}}$ along with the inclusion $E\hookrightarrow F$, give functors ${S}_{F/E}:\operatorname{Sh}(X,E)\rightarrow\operatorname{Sh}(X,F)$, which satisfy the same compatibilities as before.\\

We now construct the category of $\overline{\mathbb{Q}}_l$-sheaves denoted $\operatorname{Sh}(X,\overline{\mathbb{Q}}_l)$, in a manner analogous to the direct limit construction, using the categories $\operatorname{Sh}(X,E)$. The objects of $\operatorname{Sh}(X,\overline{\mathbb{Q}}_l)$ are $\mathscr{F}\otimes E \in \operatorname{Sh}(X,E)$, where $\mathscr{F}\otimes E$ and its images, $S_{F/E}(\mathscr{F}\otimes E)$ are identified. To define the set of morphisms between two objects, note that any two objects of $\operatorname{Sh}(X,\overline{\mathbb{Q}}_l)$ can be thought of as living in a common category $\operatorname{Sh}(X,E)$ for some $E$. We therefore define $\operatorname{Hom}_{\operatorname{Sh}(X,\overline{\mathbb{Q}}_l)}(\mathscr{F}\otimes E,\mathscr{G}\otimes E)=\varinjlim_{F/E}\operatorname{Hom}_{\operatorname{Sh}(X,F)}(S_{F/E}(\mathscr{F}\otimes E),S_{F/E}(\mathscr{G}\otimes E))$, where the colimit is taken over all finite extensions.
    
\end{subsubsection}
\begin{subsection}{The Derived category of $\overline{\mathbb{Q}}_l$-sheaves}

\begin{defn}
 We denote by $D_c^b(X,\mathfrak{o}_i)$ the full additive subcategory of the bounded derived category of $\operatorname{Sh}(X,\mathfrak{o}_i)$ with constructible cohomology sheaves.
\end{defn}
\begin{defn}
    A sheaf $\mathscr{F}$ of $\mathfrak{o}_i$-modules is called \textit{flat} if for each geometric point, the stalk is a free $\mathfrak{o}_i$-module. A \textit{perfect complex} is a complex $\mathscr{K}^\bullet\in D_c^b(X,\mathfrak{o}_i)$, which is bounded ($\mathscr{K}^n=0$ for $n$ large or small enough) and each $\mathscr{K}^n$ is a constructible flat sheaf of $\mathfrak{o}_i$-modules. We denote by $D^b_{ct f}(X,\mathfrak{o}_i)$ the full subcategory of complexes in $D_c^b(X,\mathfrak{o}_i)$ that are isomorphic to a perfect complex.
\end{defn}
\paragraph{Construction of $D_c^b(X,\mathfrak{o})$.}We first introduce the right exact functor $-\otimes_{\mathfrak{o}_{i+1}}\mathfrak{o}_i: \operatorname{Sh}(X,\mathfrak{o}_{i+1})\rightarrow \operatorname{Sh}(X,\mathfrak{o}_{i})$, for which the class of flat sheaves is an acyclic class. Therefore, the left derived functor $-\otimes_{\mathfrak{o}_{i+1}}^\mathbf{L}\mathfrak{o}_i:D^b(X,\mathfrak{o}_{i+1})\rightarrow D^b(X,\mathfrak{o}_{i})$ exists, and can be computed using a resolution of flat sheaves.\\

We take the objects of $D_c^b(X,\mathfrak{o})$ to be sequences of complexes $\{\mathscr{K}^\bullet_i\}_{i\geq1}$, such that $\mathscr{K}_i\in D^b_{ct f}(X,\mathfrak{o}_i)$, together with isomorphisms $\phi^{\mathscr{K}}_i:\mathscr{K}^\bullet_{i+1}\otimes_{\mathfrak{o}_{i+1}}^{\mathbf{L}}\mathfrak{o}_i\xrightarrow{\backsimeq}\mathscr{K}^\bullet_{i+1}$. The morphisms between two sequences are morphisms in each level that satisfy the obvious compatibility conditions.\\

 Notice now, that we have a canonical morphism $\mathscr{F}^\bullet_{i+1}\rightarrow\mathscr{F}^\bullet_{i+1}\otimes_{\mathfrak{o}_{i+1}}\mathfrak{o}_i$, which gives us an induced morphism in cohomology $\mathcal{H}^n(\mathscr{F})\rightarrow \mathcal{H}^n(\mathscr{F}^\bullet_{i+1}\otimes_{\mathfrak{o}_{i+1}}\mathfrak{o}_i)$. Using compatible flat resolutions, which exists because of \cite{kiehl2013weil}[lemma~II.5.3], one can arrange that the morphism $\mathscr{F}^\bullet_{i+1}\rightarrow\mathscr{F}^\bullet_{i+1}\otimes_{\mathfrak{o}_{i+1}}\mathfrak{o}_i\cong\mathscr{F}^\bullet_{i+1}\otimes_{\mathfrak{o}_{i+1}}^\mathbf{L}\mathfrak{o}_i\cong\mathscr{F}^\bullet_i$ is an honest map of complexes rather than just being a map in the derived category. This gives us an associated projective system $\{\mathcal{H}^n(\mathscr{K^\bullet_i})\}_{i\geq1}$ for which we have the following result, which also justifies our use of the adjective  "bounded" in $D_c^b(X,\mathfrak{o})$.

\begin{prop}\cite{kiehl2013weil}[Lemma~II.5.5]\label{2.17}
  The projective system $\{\mathcal{H}^n(\mathscr{K^\bullet_i})\}_{i\geq1}$ is an A-R $\pi$-adic sheaf. Moreover, it is zero for $|n|\gg n_0$ . 
\end{prop}
\end{subsection}
\begin{subsubsection}{Construction of $D_c^b(X,\overline{\mathbb{Q}}_l)$}
In a similar manner to before, we first construct $D^b_c(X,E)$. Its objects are those of $D^b_c(X,\mathfrak{o})$. We write again, $\mathscr{K}\otimes E$ when viewing $\mathscr{K}\in D^b_c(X,\mathfrak{o})$ as an object in $D^b_c(X,E)$. Note once again that since $\operatorname{Hom}_{D_c^b(X,\mathfrak{o}_i)}(\mathscr{F}_i,\mathscr{G}_i)$ is an $\mathfrak{o}_i$-module, we have that $\operatorname{Hom}_{D_c^b(X,\mathfrak{o})}(\mathscr{F},\mathscr{G})$ is an $\mathfrak{o}$-module. Once more, we use the localization procedure to define the homomorphism set.\\

For a finite extension $F/E$, we wish to construct functors $T_{\mathfrak{O}/\mathfrak{o}}:D^b_c(X,\mathfrak{o})\rightarrow D^b_c(X,\mathfrak{O})$. We first construct a functor $\widetilde{T}_{\mathfrak{O}/\mathfrak{o}}:D^b_c(X,\mathfrak{o})\rightarrow D^b_c(X,\mathfrak{O})^e$, where the latter category is defined by the property that for $\mathscr{K}^\bullet\in D^b_c(X,\mathfrak{O})^e$ we have $\mathscr{K}^\bullet_i\in D^b_{c tf}(X,\mathfrak{O}_{ie})$. We set $(\widetilde{T}_{\mathfrak{O}/\mathfrak{o}}(\mathscr{K}^\bullet))_i=\mathscr{K}\otimes^\mathbf{L}_{\mathfrak{o}_i}\mathfrak{O}_{ie}$. In \cite{kiehl2013weil}[II.5.5] we see that the structure morphisms are well defined and that this extends to a functor $T_{\mathfrak{O}/\mathfrak{o}}:D^b_c(X,\mathfrak{o})\rightarrow D^b_c(X,\mathfrak{O})$. Using the same localization procedure, we obtain functors $T_{F/E}:D^b_c(X,E)\rightarrow D^b_c(X,F)$ that again satisfy the usual compatibilities. Once more we define $D^b_c(X,\overline{\mathbb{Q}}_l)$ using the same direct limit procedure as in \ref{2.2.3}.
\end{subsubsection}
\begin{subsubsection}{Cohomology sheaves}\label{2.3.2}
We have seen through proposition \ref{2.17}, that for $\mathscr{F}^\bullet\in D^b_c(X,\mathfrak{o})$ we have well defined $\text{A-R}\  \pi$-adic sheaves $\{\mathcal{H}^n(\mathscr{F}^\bullet_i)\}_{i\geq0}\in\operatorname{Sh_{A-R}}(X,\mathfrak{o})$. We will now extend this construction to $D^b_c(X,\overline{\mathbb{Q}}_l)$.\\

We may without loss of generality assume that our cohomology sheaves $\{\mathcal{H}^n(\mathscr{F}^\bullet_i)\}_{i\geq0}$ are in fact honest  $\pi$-adic sheaves. We then have a natural isomorphism of functors $\mathcal{H}^n(T_{\mathfrak{O}/\mathfrak{o}}(-))\cong S_{\mathfrak{O}/\mathfrak{o}}(\mathcal{H}^n(-))$. This is due to the fact that $-\otimes_{\mathfrak{o}_i}\mathfrak{O}_{ie}$ is exact. Indeed, $\mathfrak{O}$ is a free $\mathfrak{o}$-module, and therefore $\mathfrak{O}_{ie}$ is also a free $\mathfrak{o}_{i}$-module.\\

Notice now that the localization procedure for both $D^b_c(X,\mathfrak{o})$ and $\operatorname{Sh}(X,\mathfrak{o})$ both only involve tensoring the morphism set by $E$. We can therefore upgrade the previous construction to the functor $\mathcal{H}^n:D^b_c(X,E)\rightarrow\operatorname{Sh}(X,E)$. The previous natural isomorphism now becomes, $\mathcal{H}^n(T_{F/E}(-))\cong S_{F/E}(\mathcal{H}^n(-))$. This compatibility allows us to define cohomology functors 
\begin{align*}
    \mathcal{H}^n:D^b_c(X,\overline{\mathbb{Q}}_l)\rightarrow\operatorname{Sh}(X,\overline{\mathbb{Q}}_l).
\end{align*}
Furthermore, proposition \ref{2.17} translates into the fact that for $\mathscr{F}\otimes E\in D^b_c(X,\overline{\mathbb{Q}}_l)$, we have that the cohomology sheaves $\mathcal{H}^n(\mathscr{F}\otimes E)$ are zero in large enough degrees.
\end{subsubsection}
\begin{subsubsection}{The derived functors}
We finish this section by introducing the derived functors associated to the derived category just constructed. These functors relate to each other using the usual adjoint identities. Their construction follows a similar method to what we have seen before. One key part of the process, occurs at the level of the derived category $D^b_{ct f}(X,\mathfrak{o}_i)$, where one has to prove that the the associated derived functors behave well with respect to constructible sheaves. \\

Given a morphism $f:X\rightarrow Y$, we have the following functors,
\begin{itemize}
    \item $\mathbf{R}f_*:D^b_c(X,\overline{\mathbb{Q}}_l)\rightarrow D^b_c(Y,\overline{\mathbb{Q}}_l)$,
    \item $f^*:D^b_c(Y,\overline{\mathbb{Q}}_l)\rightarrow D^b_c(X,\overline{\mathbb{Q}}_l)$
    \item $\mathbf{R}f_!:D^b_c(X,\overline{\mathbb{Q}}_l)\rightarrow D^b_c(Y,\overline{\mathbb{Q}}_l)$
    \item $f^!:D^b_c(Y,\overline{\mathbb{Q}}_l)\rightarrow D^b_c(X,\overline{\mathbb{Q}}_l)$
    \item $\mathbf{R}\mathscr{H}om:D^b_c(X,\overline{\mathbb{Q}}_l)^{op}\times D^b_c(X,\overline{\mathbb{Q}}_l)\rightarrow D^b_c(X,\overline{\mathbb{Q}}_l)$
    \item $-\otimes^\mathbf{L}-:D^b_c(X,\overline{\mathbb{Q}}_l)\times D^b_c(X,\overline{\mathbb{Q}}_l)\rightarrow D^b_c(X,\overline{\mathbb{Q}}_l)$
\end{itemize}
These functors satisfy the following adjoint relationships,
\begin{align*}
    \operatorname{Hom}_{D^b_c(X,\overline{\mathbb{Q}}_l)}(f^*\mathscr{K},\mathscr{L})&\cong{Hom}_{D^b_c(X,\overline{\mathbb{Q}}_l)}(\mathscr{K},\mathbf{R}f_*\mathscr{L})\\
    \operatorname{Hom}_{D^b_c(X,\overline{\mathbb{Q}}_l)}(\mathbf{R}f_!\mathscr{K},\mathscr{L})&\cong{Hom}_{D^b_c(X,\overline{\mathbb{Q}}_l)}(\mathscr{K},f^!\mathscr{L}).\\
\end{align*}
\end{subsubsection}
\end{subsection}

\begin{subsection}{The Function-Sheaf Dictionary}
    We now introduce a rough correspondence between complexes in $D^b_c(X,\overline{\mathbb{Q}}_l)$ and functions on the $\mathbb{F}_q$ points of an algebraic variety $X$. This correspondence is not a bijection, but rather serves to give good intuition by relating sheaves to more classical notions.

\begin{subsubsection}{$\overline{\mathbb{Q}}_l$-sheaves and representations}
\begin{defn} The \textit{stalk} $\mathscr{F}_{\overline{x}}$,  of a $\pi$-adic sheaf $\mathscr{F}\in\operatorname{Sh}(X,\mathfrak{o})$ at a geometric point $\overline{x}$, is defined to be the finite $\mathfrak{o}$-module $\varprojlim(\mathscr{F}_n)_{\overline{x}}$.   
\end{defn}
Using theorem \ref{2.11}, and the fact that any finite $\mathfrak{o}$-module $M$ can be represented as $\varprojlim M_n$, where $M_n=M/\mathfrak{o}^{m+1}M$, we have an equivalence of categories between $\pi$-adic sheaves and finite $\mathfrak{o}$-modules with a continuous action of $\pi_1(X,\overline{x})$.
\begin{defn}The \textit{stalk} $(\mathscr{F}\otimes E )_{\overline{x}}$, of an $E$-sheaf $\mathscr{F}\otimes E\in\operatorname{Sh}(X,E)$ at a geometric point $\overline{x}$, is defined to be the finite dimensional $E$-vector space $E\otimes\mathscr{F}_{\overline{x}}$.
\end{defn}
\begin{prop}
We have an equivalence of categories between $E$-sheaves on $X$ and continuous representations of $\pi_1(X,\overline{x})$ on finite dimensional $E$-vector spaces.
\end{prop}

We reduce this case to the previous result. Let $V$ be an $E$-vector space with an action of $\pi_1(X,\overline{x})$. Consider any finite $\mathfrak{o}$-module $L_0$ that spans $V$. Notice that since $L_0$ is open in $V$, $\pi_1(X,\overline{x})\times L_0$ is open in $\pi_1(X,\overline{x})\times V$. Let $\alpha:\pi_1(X,\overline{x})\times V\rightarrow V$ be the action morphism. Again since $L_0$ is open in $V$, we have that $\alpha^{-1}(L_0)$ is open in $\pi_1(X,\overline{x})\times V\rightarrow V$. Therefore the intersection of $\pi_1(X,\overline{x})\times L_0$ and $\alpha^{-1}(L_0)$ is also open. The open subgroup $H$, that corresponds to the first factor of this open set, takes $L_0$ to itself. Furthermore, since $\pi_1(X,\overline{x})$ is compact, $H$ is a finite index subgroup. If we now set $L=\sum g(L_0)$, where $g$ runs over the finite number of representatives of $\pi_1(X,\overline{x})/H$, we get a $\pi_1(X,\overline{x})$-stable $\mathfrak{o}$-submodule.

\begin{defn}
    The \textit{stalk} $(\mathscr{F}\otimes E )_{\overline{x}}$, of an $\overline{\mathbb{Q}}_l$-sheaf $\mathscr{F}\otimes E\in\operatorname{Sh}(X,\overline{\mathbb{Q}}_l)$ at a geometric point $\overline{x}$, is defined to be the finite dimensional $\overline{\mathbb{Q}}_l$-vector space $(\mathscr{F}\otimes E )_{\overline{x}}\otimes\overline{\mathbb{Q}}_l$.
\end{defn}
\begin{prop}\label{2.22}
We have an equivalence of categories between $\overline{\mathbb{Q}}_l$-sheaves on $X$ and continuous representations of $\pi_1(X,\overline{x})$ on finite dimensional $\overline{\mathbb{Q}}_l$-vector spaces.
\end{prop}

The essential part of the proof is to show that any representation $\rho:\pi_1(X,\overline{x})\rightarrow GL_n(\overline{\mathbb{Q}}_l)$ factors over $GL_n(K)$ for some finite extension $K/\mathbb{Q}_l$. The group $GL_n(K)$ is Hausdorff since $GL_n$ is separated. Furthermore the image of $\rho$ which we call $\Gamma$ is compact. Therefore, one can use the Baire category theorem to show that one of the interiors of $\Gamma\cap GL_n(K)$ is non-empty, which suffices for the claim.
    
\end{subsubsection}
\begin{subsubsection}{Constructing functions from sheaves.}\label{2.4.2}
For the rest of this section, let $X_0$ be a variety over $\mathbb{F}_q$ and $K_0\in D^b_c(X,\overline{\mathbb{Q}}_l)$. Let $X=X_0\times_{\mathbb{F}_q}\overline{\mathbb{F}}_q$ and $K$ be the pullback of $K_0$ to $X$. Where convenient, we refer to $\mathbb{F}_q$ and $\overline{\mathbb{F}}_q$ as $k$ and $\overline{k}$ respectively. We can then associate a function to $K_0$,

    $$ f^{K_0}:X_0(\mathbb{F}_q)\rightarrow \overline{\mathbb{Q}}_l$$
    which for a  point $x \in X_0(\mathbb{F}_q)$ is given by
    $$\sum_i(-1)^i \Tr(\mathrm{Fr},(\mathcal{H}^i(K_0))_{\overline{x}} )$$

Here $\overline{x}$ is a geometric point over $x\in X_0$, and $\mathrm{Fr}\in\mathrm{Gal}(k(\overline{x}),k(x))$ is the Frobenius element in the Galois group. The notation $\mathcal{H}$ is the same as from section \ref{2.3.2}, from which we have that the sum is finite. Recall from the same, that the cohomology sheaves are $\overline{\mathbb{Q}}_l$-sheaves. Hence, by proposition \ref{2.22} we get an action of $\pi_1(x,\overline{x})=\mathrm{Gal}(k(\overline{x}),k(x))$. In particular $\mathrm{Fr}\in\mathrm{Gal}(k(\overline{x}),k(x))$ corresponds to an element in $GL_n(\overline{\mathbb{Q}}_l)$ for which we can take the trace.

\begin{rmk}
    We can also associate a closely related function $f_g^{K_0}$, by considering the action of the inverse of the Frobenius morphism in $\mathrm{Gal}(k(\overline{x}),k(x))$. We call this element the geometric Frobenius, and use the subscript to distinguish between the two cases. We will be using this function instead of the one described above for some calculations later.
\end{rmk}
\begin{defn}\label{2.25}
    We denote by $\mathrm{Fr}_g$ the action of the geometric Frobenius induced by functionality on $H^i_c(X,\mathcal{F})$ in the following way. Let $f^{-1}\in \mathrm{Gal}(\overline{\mathbb{F}}_q,\mathbb{F}_q)$ represent the inverse Frobenius morphism. Then $f^{-1}$ induces an endomorphism of $\mathrm{Fr}_g=:\mathrm{id}_X\times f^{-1}:X\rightarrow X$. If we denote by $p:X\rightarrow X_0$, the projection morphism, we have the following chain of isomorphisms for a sheaf $\mathcal{F}_0$ on $X_0$,
    $$\mathrm{Fr}_g^*(\mathcal{F})=\mathrm{Fr}_g^*(p^*\mathcal{F}_0)=p^*\mathcal{F}_0=\mathcal{F}.$$

    Therefore we have an induced action on cohomology;
$$H^i_c(X,\mathcal{F})\rightarrow H^i_c(X,\mathrm{Fr}_g^*(\mathcal{F}))\simeq H^i_c(X,\mathcal{F})$$
\end{defn}

\begin{rmk}
We now offer a different perspective on this action which will be useful in later applications. We can identify the pullback of a sheaf to a geometric point, with its global sections. Therefore, the base change theorem tells us that $H^i_c(X,\mathcal{F})\cong (R^i f_!(X_0, {\mathcal{F}}_0))_{\overline{k}}$. let $f\in \mathrm{Gal}(\overline{k},k)$ represent the Frobenius automorphism. We then have the following diagram of étale neighborhoods,
\[\begin{tikzcd}
	{k'} & {k'} \\
	{\overline{k}}
	\arrow["{f^{-1}|_{k'}}", from=1-1, to=1-2]
	\arrow["{\overline{u}\circ f}", from=2-1, to=1-1]
	\arrow["{\overline{u}}"', from=2-1, to=1-2]
\end{tikzcd}\]

Notice now that the sheaf $R^if_!(X_0,\mathcal{F}_0)$ is given as the sheafification of the presehaf $k'\mapsto H^k_c(X_{k'},\mathcal{F}_{k'})$, and that the stalk of the sheaf is isomorphic to the stalk of the presheaf. We therefore have that the above map induces a morphism of global sections in the colimit which defines the stalk.

\[\begin{tikzcd}
	{H^k_c(X_{k'},\mathcal{F}_{k'})^{\overline{u}}} && {H^k_c(X_{k'},\mathcal{F}_{k'})^{\overline{u}\circ f}}
	\arrow["{t\mapsto\alpha(t)}", from=1-1, to=1-3]
\end{tikzcd}\]

where $\alpha:H^k_c(X_{k'},\mathcal{F}_{k'})\rightarrow H^k_c(X_{k'},\mathcal{F}_{k'})$ is the analogous action on cohomology induced by the geometric Frobenius element through functionality. The superscripts simply keep track of the indices in the colimit.\\

Recall now that the action of $\mathrm{Gal}(k(\overline{x})/k(x))$ described in definition \ref{2.12} simply permuted the indices in the colimit, by precomposition with $f$. This gives us the following diagram,

\[\begin{tikzcd}
	{H^k_c(X_{k'},\mathcal{F}_{k'})^{\overline{u}}} && {H^k_c(X_{k'},\mathcal{F}_{k'})^{\overline{u}\circ f}} \\
	{H^k_c(X_{k'},\mathcal{F}_{k'})^{\overline{u}}}
	\arrow["{t\mapsto\alpha(t)}", from=1-1, to=1-3]
	\arrow[dashed, from=1-1, to=2-1]
	\arrow["{\alpha(t)\mapsto\alpha(t)}", from=1-3, to=2-1]
\end{tikzcd}\]

Therefore, we see that the total effect of the action of  $f\in\mathrm{Gal}(k(\overline{x})/k(x))$ on $(R^i f_!(X_0, \mathcal{F}_0))_{\overline{k}}$ is the same as that of the action previously described by functionality in definition \ref{2.25}. This alternative perspective of Galois actions will be useful in the the following theorem.\\

\end{rmk}

\begin{thm}[Grothendieck trace formula]\label{2.23}
    With the same notation as before and $\mathcal{F}\in D^b_c(X,\overline{\mathbb{Q}}_l)$, we have the equality,
    $$\sum_{i\in\mathbb{Z}}\sum_{x\in X_0(\mathbb{F}_q)}(-1)^i\Tr(\mathrm{Fr}_g,(\mathcal{H}^i\mathcal{F}_0)_{\overline{x}})=\sum_i(-1)^i\Tr(\mathrm{Fr}_g,{H}^i_c(X,\mathcal{F}))$$
\end{thm}
Here the action on the right hand side is as described in definition \ref{2.25}. Whereas, the action of $\mathrm{Fr}_g=\mathrm{Fr^{-1}} \in \mathrm{Gal}(k(\overline{x})/k(x))=\pi_1(x,\overline{x})$ on $(\mathcal{H}^i \mathcal{F}_0)_{\overline{x}}$ is the one outlined in theorem \ref{2.11}.

\end{subsubsection}   
    
\end{subsection}

\section{The Fourier Transform for the Affine Line}
In this section we detail the construction of the Fourier transform on the affine line. We also calculate some key examples and try to relate it to the classical Fourier transform for finite groups. Throughout this section we fix a prime $l\neq p$, and a geometric point $\overline{x}\rightarrow\mathbb{A}_0$ over the origin.

\begin{subsection}{The Set-up}

Let $k$ be a finite field of characteristic $p$, and let $\mathbb{A}_0$ denote the affine line over $k$. The Artin-Schreier covering given by;
 \begin{align*}
      \mathscr{P}:\mathbb{A}_0  &\rightarrow \mathbb{A}_0\\
 x  &\rightarrow x^q-x
 \end{align*}
 is a finite étale Galois cover, whose Galois group is canonically isomorphic to the additive group of $k$. This implies that $k$ is a quotient of the étale fundamental group $\pi_1(\mathbb{A}_0,\overline{x})$. Therefore any character,

 $$\psi: k\rightarrow\overline{\mathbb{Q}}_l^*$$ gives by precomposition, a character of the étale fundamental group. By proposition \ref{2.22} which gives an equivalence of categories between $\overline{\mathbb{Q}}_l$  sheaves and continuous representations of the étale fundamental group on finite dimensional $\overline{\mathbb{Q}}_l$-vector spaces, we get an associated sheaf $$\mathcal{L}_0(\psi)\in D^b_c(\mathbb{A}_0,\overline{\mathbb{Q}}_l).$$

\begin{defn} We fix $\psi: k \rightarrow \overline{\mathbb{Q}}_l^*$ a nontrivial character. Consider the diagram

\[\begin{tikzcd}
	& {\mathbb{A}_0\times\mathbb{A}_0} && {\mathbb{A}_0} \\
	\\
	{\mathbb{A}_0} && {\mathbb{A}_0}
	\arrow["m"', from=1-2, to=1-4]
	\arrow["{\pi_1}", from=1-2, to=3-1]
	\arrow["{\pi_2}"', from=1-2, to=3-3]
\end{tikzcd}\]

\end{defn}

we define the Fourier transform,
$$T_\psi: D^b_c(\mathbb{A}_0,\overline{\mathbb{Q}}_l) \rightarrow D^b_c(\mathbb{A}_0,\overline{\mathbb{Q}}_l)$$by
$$T_\psi(K_0)=R {\pi_1}_!({\pi_2}^*(K_0)\otimes m^*\mathcal{L}_0(\psi))[1]$$

Where $\pi_1$ and $\pi_2$ correspond to the two projections and $m(x,y)=xy$ is the multiplication map.
\begin{defn}
    Let $G$ be a finite abelian group and $\widehat{G}$ its group of characters. let us fix a character $\psi:G\rightarrow\mathbb{C}$, which gives us an isomorphism $G\xrightarrow{\cong}\widehat{G}$, by sending $g\in G$ to $\psi_g$, where $\psi_g(\gamma)=\psi(g\gamma)$. Let $f:G\rightarrow\mathbb{C}$ be a function. Using the prescribed isomorphism, we define the Fourier transform of the function $f$ to be $FT(f):G\rightarrow\mathbb{C}$, defined by,

    $$FT(f)(g)=\sum_{\gamma\in G}f(\gamma)\overline{\psi_g(\gamma)}$$
    note here that $\overline{\psi_g}=\psi_g^{-1}$.
\end{defn}

We will now prove that (up to a sign) the Fourier transform on the affine line respects the Function-Sheaf dictionary we outlined in section \ref{2.4.2}, and the Fourier transform for finite abelian groups detailed above. 
\begin{thm}\label{Comparison theorem}For all $K_0 \in D^b_c(\mathbb{A}_0,\overline{\mathbb{Q}}_l)$ we have;

\begin{align}
    f_g^{T_\psi K_0}(t)=-\sum_{x\in k}f_g^{K_0}(x)\psi^{-1}_x(t) 
    \tag{$x\in k$}
\end{align}

as functions $\mathbb{A}_0(k)=k\rightarrow \overline{\mathbb{Q}}_l$.

\end{thm}

\begin{notation}    For $a\in k $, let $$\lambda_a: \mathbb{A}_0 \rightarrow\mathbb{A}_0$$ $$x\mapsto ax$$ 

\end{notation}
\begin{lem}
    For a character $\psi: k \rightarrow \overline{\mathbb{Q}}_l^* $ we have
    $$ \lambda_a^*(\mathcal{L}_0(\psi))=\mathcal{L}_0(\psi_a)$$

\end{lem}

\begin{lem}\label{3.6}
    Let $\overline{k}$ be an algebraic closure of $k$. Given a complex $K_0 \in D^b_c(\mathbb{A}_0,\overline{\mathbb{Q}}_l)$ and a geometric point $a \in \mathbb{A}_0(\overline{k})$, we have
    \\
    $$\Bigl( T_\psi(K_0)\Bigr)_a=R\Gamma_c \Bigl( (K) \otimes \mathcal{L}(\psi_a)\Bigr)[1] $$
\end{lem}

\begin{lem}\label{3.7}
We have the following identities 
\begin{align*}
    f^{\mathcal{L}(\psi)}(x)&=\psi(x)\\
f_g^{\mathcal{L}(\psi)}(x)&=\psi^{-1}(x)
\end{align*}
\end{lem}

   We now move on to the proof of \ref{Comparison theorem}. We have the following chain of equalities;
   \begin{align*}
f_g^{T_\psi K_0}(t)&
=\sum(-1)^i\Tr(\mathrm{Fr}_g,\mathcal{H}^i(T_\psi K_0)_{\overline{t}})\\
&=\sum(-1)^i\Tr(\mathrm{Fr}_g,\mathcal{H}^i((T_\psi K_0)_{\overline{t}}))\\
&\overset{\ref{3.6}}{=}
\sum(-1)^i\Tr(\mathrm{Fr}_g,\mathcal{H}^i(R\Gamma_c(K)\otimes\mathcal{L}(\psi_t)[1]))\\
&=\sum(-1)^i\Tr(\mathrm{Fr}_g,H^i_c((K)\otimes\mathcal{L}(\psi_t))[1]))\\
&\overset{\ref{2.23}}{=}\sum_{i\in\mathbb{Z}}\sum_{x\in X_0(\mathbb{F}_q)}(-1)^i\Tr(\mathrm{Fr}_g,((K_0)\otimes\mathcal{L}_0(\psi_t))_{\overline{x}}[1])\\
&\overset{\ref{3.7}}{=}-\sum_{i\in\mathbb{Z}}\sum_{x\in X_0(\mathbb{F}_q)}(-1)^i\Tr(\mathrm{Fr}_g,\mathcal{H}^i(K_0)_x))\psi(-xt)\\
&=-\sum_{x\in X_0(\mathbb{F}_q)}f^{K_0}(x)\psi^{-1}_x(t)
   \end{align*}

\end{subsection}

\begin{subsection}{Properties of the Fourier transform}
We now prove various properties of the Fourier transform that closely align with the classical case. These also give some intuition as to what to expect in the cases of interest in the future. The main theorem regarding the Fourier transform is the following.

\begin{notation}
We will denote by $\overline{\mathbb{Q}}_l(1)$ the Tate twist of the sheaf $\overline{\mathbb{Q}}_l$. This is constructed as the projective limit of $\{\mu_{p^n}\}_{n\geq 0}$ of sheaves of roots of unity. The sheaf $\overline{\mathbb{Q}}_l(d)$ is obtained as the $d$-fold tensor product of $\overline{\mathbb{Q}}_l(1)$, where negative numbers denote the dual sheaf. For a sheaf $K_0 \in D^b_c(\mathbb{A}_0,\overline{\mathbb{Q}}_l)$ we denote the $d^{th}$ Tate twist by,
$$K_0(d):=K_0\otimes_{\overline{\mathbb{Q}}_l}\overline{\mathbb{Q}}_l(d)$$.
\end{notation}
\begin{thm}[Fourier Inversion]

For all $K_0 \in D^b_c(\mathbb{A}_0,\overline{\mathbb{Q}}_l)$ we have a canonical isomorphism;
$$ \Bigl(T_{\psi^{-1}} \circ  T_\psi\Bigr)(K_0)= K_0(-1)$$

\end{thm}

We will not be proving this theorem here, although we would like to highlight a key ingredient that is used in the proof on the theorem. Namely, we would like to show that the Fourier transform of the constant sheaf $\overline{\mathbb{Q}}_l$ on $\mathbb{A}_0$ is essentially the skyscraper sheaf at the origin (up to shifts and twists).

Consider again the mapping $$k \rightarrow Hom(k,\overline{\mathbb{Q}}_l^{\times})$$ given by $$x\mapsto\psi_x$$ for a fixed character $\psi$. the structure theorem for finite abelian groups gives us that this is an isomorphism.\\

Now let us consider the pushforward of the constant sheaf $\overline{\mathbb{Q}}_l$ of $\mathbb{A}_0$ along the Artin-Schreier map, $\mathscr{P}_*(\overline{\mathbb{Q}}_l)$. By the correspondence of finite étale covers and LCC sheaves, one sees that the associated representation is the regular representation of $k$. This implies

    \begin{align}\label{1.0}
        \mathscr{P}_*(\overline{\mathbb{Q}}_l)=\bigoplus_{x\in k}\mathcal{L}_0(\psi_x)
    \end{align}

\begin{lem}
We also have the following calculations from étale cohomology.

$${H}^1(\mathbb{A},\mathbb{Q}_l)={H}^1(\mathbb{A},\mathscr{P}_*(\overline{\mathbb{Q}}_l))=0$$ and
$${H}^1_c(\mathbb{A},\mathbb{Q}_l)={H}^1_c(\mathbb{A},\mathscr{P}_*(\overline{\mathbb{Q}}_l))=0.$$
\end{lem}

This along with the identity (\ref{1.0}), gives us that
$${H}^1(\mathbb{A},\mathcal{L}_0(\psi_x))={H}^1_c(\mathbb{A},\mathcal{L}_0(\psi_x))=0.$$

Since $\mathbb{A}$ is affine, Artin vanishing gives us,
$$H_c^0(\mathbb{A},\mathcal{L}_0(\psi_x))=0.$$

\begin{rmk}
    Note that by the description of the action of $G_x:=\mathrm{Gal}(k(\overline{x})/k(x))$ on the stalk $\mathcal{F}_{\overline{x}}$ as outlined in remark \ref{2.13}, we see that the image of $\Gamma(X,F)\in \mathcal{F}_{\overline{x}}$ is contained in the set of elements invariant under the Galois action, i.e $\Gamma(X,F) \subset \mathcal{F}_{\overline{x}}^{G_x}$.
\end{rmk}

By construction we have that stalk of $\mathcal{L}(\psi_x)$ corresponds to the one dimensional representation of $\pi_1(\mathbb{A},\overline{x})$ given by $\psi_x$. This is a nontrivial one dimensional representation whenever $x\neq 0$. In particular this implies by the previous remark, $\Gamma(\mathbb{A},\mathcal{L}(\psi_x)) \subset \mathcal{L}(\psi_x)_{\overline{x}}^{G_x}=\{0\}$. Stated another way $$H^0(\mathbb{A},\mathcal{L}(\psi_x))=0$$
whenever $x\neq0$. Therefore, by Poincaré duality, we have

\begin{equation*}
 {H}_c^2(\mathbb{A},\mathcal{L}_0(\psi_x))= 
\left\{
    \begin{array}{lr}
        0, & \text{if } x \neq 0\\
        \overline{\mathbb{Q}}_l(-1), & \text{if } x= 0
    \end{array}
    \right.
\end{equation*}

\begin{lem}[Orthogonality relations] Let $i_0:\{0\}\hookrightarrow \mathbb{A}_0$ be the inclusion of the origin. We then have a canonical isomorphism,
    $$T_\psi(\overline{\mathbb{Q}}_l)\cong {i_{0}}_*\overline{\mathbb{Q}}_l(-1)[-1].$$

\end{lem}
\begin{proof}
        First we construct a canonical map between the two sheaves. Consider the diagram,
\[\begin{tikzcd}
	& {\{0\}\times\mathbb{A}_0} && {\mathbb{A}_0\times\mathbb{A}_0} && {\mathbb{A}_0} \\
	\\
	{\{0\}} && {\mathbb{A}_0} && {\mathbb{A}_0}
	\arrow["j", from=1-2, to=1-4]
	\arrow["{\widetilde\pi_1}"', from=1-2, to=3-1]
	\arrow["\lrcorner"{anchor=center, pos=0.125}, draw=none, from=1-2, to=3-3]
	\arrow["m", from=1-4, to=1-6]
	\arrow["{\pi_1}", from=1-4, to=3-3]
	\arrow["{\pi_2}"', from=1-4, to=3-5]
	\arrow["{i_0}"', from=3-1, to=3-3]
\end{tikzcd}\]

where $\pi_1$ and $\pi_2$ are the canonical projections, and $m$ is the multiplication map. The square on the left is cartesian, equipped with canonical projections $j $ and $\widetilde{\pi}_1$. We therefore have,

\begin{align*}
    &\hspace{15pt}Hom(R {\pi_1}_!({\pi_2}^*(\overline{\mathbb{Q}}_l)\otimes m^*\mathcal{L}_0(\psi))[1],(i_0)_*\overline{\mathbb{Q}}_l)\\
   & \simeq Hom(i^*R {\pi_1}_!({\pi_2}^*(\overline{\mathbb{Q}}_l)\otimes m^*\mathcal{L}_0(\psi))[1],\overline{\mathbb{Q}}_l)\\\tag{base change formula}
   & \simeq Hom(\widetilde{\pi}_{1*}j^*({\pi_2}^*(\overline{\mathbb{Q}}_l)\otimes m^*\mathcal{L}_0(\psi))[1],\overline{\mathbb{Q}}_l)\\
   & \simeq Hom(\widetilde{\pi}_{1*}j^*({\pi_2}^*(\overline{\mathbb{Q}}_l))\otimes j^* (m^*\mathcal{L}_0(\psi)))[1],\overline{\mathbb{Q}}_l)\\
   & \simeq Hom(\widetilde{\pi}_{1*}(\overline{\mathbb{Q}}_l)\otimes j^* (m^*\mathcal{L}_0(\psi)))[1],\overline{\mathbb{Q}}_l)\\\tag{$ m\circ j=\lambda_0$}
   & \simeq Hom(\widetilde{\pi}_{1*}(\overline{\mathbb{Q}}_l) )[1],\overline{\mathbb{Q}}_l)\\
   & \simeq Hom(\overline{\mathbb{Q}}_l [1],\overline{\mathbb{Q}}_l)\\
   &\simeq Hom(\overline{\mathbb{Q}}_l,\overline{\mathbb{Q}}_l) [1]
\end{align*}

which gives a canonical map corresponding to the identity element in $Hom(\overline{\mathbb{Q}}_l,\overline{\mathbb{Q}}_l)$.
\begin{align*}
   (T_\psi(\overline{\mathbb{Q}}_l))_a&=R\Gamma_c(\overline{\mathbb{Q}}_l\otimes\mathcal{L}(\psi_a))[1]\\
   &=R\Gamma_c(\mathcal{L}(\psi_a))[1]
\end{align*}
which agrees with the stalk at $a$ for ${i_{0}}_*\overline{\mathbb{Q}}_l(-1)[-1]$ by our previous calculations.

\end{proof}

\end{subsection}
\pagebreak

\section{Duality for unipotent groups}
A good notion of duality is key to generalizing this theory. We shall first discuss the case of the Serre dual for commutative connected unipotent group schemes, whose first incarnations can be found in \cite{serre1960groupes}. Throughout this section we fix a prime $p$.

\begin{subsection}{The commutative case}

We fix a perffect field $k$. Recall that any scheme $X$ over $k$ comes equipped with an absolute Frobenius morphism $F_X$, which is  the identity on the underlying topological space, and the map on the structure sheaves $F^\sharp_X:\mathcal{O}_X\rightarrow\mathcal{O}_X$, is given by $x\mapsto x^p$.
\begin{subsubsection}{Perfect schemes}
    \begin{defn}
        A scheme  $X$ over $k$ is said to be \textit{perfect} if the absolute Frobenius morphism is an isomorphism.
        \end{defn}

    Let $\mathbf{Sch}_k$ be the category of schemes over $k$. We call $\mathfrak{Perf}_k$  the full subcategory of perfect schemes. The inclusion of $\mathfrak{Perf}_k\hookrightarrow\mathbf{Sch}_k $ has a right adjoint $X\rightarrow X^\mathrm{perf}$ called the perfectization functor.

\begin{rmk}
    Because the fiber product of perfect schemes is perfect,  a group object in $\mathfrak{Perf}_k$ is automatically a group object in $\mathfrak{Sch}_k$, and the notion of "a perfect group scheme over $k$" makes sense. Furthermore, since the perfectization functor is a right adjoint, it preserves limits. Hence if $G$ is a group scheme over $k$, then $G^\mathrm{perf}$ is a perfect group scheme over $k$.
\end{rmk}

\begin{defn}
 A perfect unipotent group over $k$ is a perfect group scheme over $k$ which is isomorphic to the perfectization of a unipotent algebraic group over $k$.
\end{defn}

\begin{rmk}
       If $\Gamma$ is an algebraic group over $k$ we can explicitly describe the perfectization functor. We denote by $\Gamma(p^{-n})$ the pullback of $\Gamma$ along the endomorphism of $k$ given by, $a\rightarrow a^{p^{-n}}$. The absolute Frobenius morphism between $\Gamma(p^{-n})$ and $\Gamma(p^{-n+1})$ is therefore $k$ linear, and we call the limit of this diagram $\Gamma^\mathrm{perf}$. It enjoys a universal property which means that for all perfect group schemes $S$, we have ${\Gamma^\mathrm{perf}}(S)= \Gamma(S)$.
    \end{rmk}
We denote by $\mathfrak{cpu}_k$
the category of all commutative perfect unipotent groups
over k, and by $\mathfrak{cpu}_k^{\circ}\subset \mathfrak{cpu}_k$ the full subcategory formed by connected group schemes. 
\end{subsubsection}
\pagebreak
\begin{subsubsection}{The Duality functor}  \hspace{5pt}

If $\Gamma$ is an connected algebraic unipotent k-group, we define a functor $\Gamma^*$ on $\mathfrak{Perf}_k$ by setting its values on all affine perfect k-schemes $S$ to be,
$$\Gamma^*(S)=\mathrm{Ext}^1_S (\Gamma_S,\mathbb{Q}_p/ 
\mathbb{Z}_p)=\varinjlim_n \mathrm{Ext}^1_S(\Gamma_S,\mathbb{Z}/p^n\mathbb{Z})$$

where $\mathbb{Z}/p^n\mathbb{Z}$ is the discrete group scheme over $S$. We call $\omega_1(\Gamma)$ the kernel of the projection $\Gamma^\mathrm{perf}\rightarrow\Gamma$, so that it sits in an exact sequence on the big fppf site over $k$ as follows,
\begin{align}\label{exact}
    0\rightarrow\omega_1(\Gamma)\rightarrow\Gamma^\mathrm{perf}\rightarrow\Gamma\rightarrow0
\end{align}

We have that the morphism $\Gamma^\mathrm{perf}\rightarrow\Gamma$ is a universal homeomorphism, and therefore the underlying topological space of the kernel is a point, and in particular connected. This implies that $\operatorname{Hom}(\omega_1(\Gamma),\mathbb{Q}/ 
\mathbb{Z})=0$.\\

We now wish to show that $\operatorname{Ext^1}(\omega_1(\Gamma)_S,\mathbb{Z}/p^n\mathbb{Z})=0$. To this end, consider an element in this group given by
$$0\rightarrow\mathbb{Z}/p^n\mathbb{Z}\rightarrow E\rightarrow\omega_1(\Gamma)_S\rightarrow0$$ for which we have the associated sequence of groups of connected components.
$$\pi_0(\mathbb{Z}/p^n\mathbb{Z})\rightarrow\pi_0(E)\rightarrow\pi_0(\omega_1(\Gamma)_S)\rightarrow0$$

By our previous discussion we know that $\pi_0(\omega_1(\Gamma)_S)=0$ as the underlying topological space is a point. This implies that we have the following commutative diagram of group schemes
\[\begin{tikzcd}
	{\mathbb{Z}/p^n\mathbb{Z}} & E \\
	{\pi_0(\mathbb{Z}/p^n\mathbb{Z})} & {\pi_0(E)}
	\arrow[from=1-1, to=1-2]
	\arrow["{\cong}", from=1-1, to=2-1]
	\arrow[from=1-2, to=2-2]
	\arrow[two heads, from=2-1, to=2-2]
\end{tikzcd}\]

We know that the dimension of algebraic groups sum in the expected way in exact sequences. In addition to this, we know that algebraic groups over fields are separated, which implies by base change that the top morphism is a closed immersion. In particular, this implies the that if two connected components of $\mathbb{Z}/p^n\mathbb{Z}$ map to the same connected component of $E$ it would fail to be an injective map on spaces.\\

This isomorphism allows us to construct a retraction for the morphism $\mathbb{Z}/p^n\mathbb{Z}\rightarrow E$, which is what we required to show the triviality of the group of extensions in question. Now applying the above results to the long exact sequence associated to (\ref{exact}), we get that the canonical morphism $\Gamma^*(S)\rightarrow (\Gamma^\mathrm{perf})^*(S)$ is an isomorphism.\\

In the remainder of this section, we try to show that this functor is representable, exact, and involutive. This section will closely follow \cite{begueri1980dualite}.
\end{subsubsection}

\begin{subsubsection}{The case of Witt vectors}
    We denote by $W_n$ the algebraic group over $\Spec(\mathbb{F}_p)$ of Witt vectors of length $n$. It can be described using its functor of points, which sends a ring $R$ over $\mathbb{F}_p$ to the ring of truncated Witt vectors of length $n$. Consider also the Artin-Schreier isogeny which we call $\epsilon_n$,
    $$0\rightarrow \mathbb{Z}/p^n\mathbb{Z} \rightarrow W_n\xrightarrow{F-id}W_n\rightarrow0$$

    For all affine $\mathbb{F}_p$-schemes $S$, we can consider the base change of this exact sequence to S. Call this $\epsilon_{n,S}$. Pulling back $\epsilon_{n,S}$ gives us a morphism,

    $$\partial_n(S):\mathrm{End}(W_{n,S})\rightarrow \mathrm{Ext}^1_S(W_{n,S},\mathbb{Z}/p^n\mathbb{Z})$$

    By the universal property of  $W_{n,S}$ we have that $W_{n,S}(S)=W_n(S)$. This implies that for all $S$-schemes $R$, $W_{n,S}(R)$ is a module over $W_n(S)$, therefore we can view elements $a \in W_n(S)$ as elements in $\mathrm{End}(W_{n,S})$ by considering the multiplication action. Using this identification we arrive at the homomorphism,

    $$\theta_n(S):W_n(S)\rightarrow W_n^*(S)=(W_n^\mathrm{perf})^*(S)$$

    If we restrict this functor to the subcategory of perfect affine $\mathbb{F}_p$-schemes, we arrive at a functor,

    $$\theta_n:W_n^\mathrm{perf}\rightarrow (W_n^\mathrm{perf})^*.$$

    \begin{thm}\label{3.5}
        For all perfect affine $\mathbb{F}_p$-schemes $S$, the homomorphism $$\partial_1(S):\mathrm{End}_S(\mathbb{G}_{a,S})\rightarrow \mathrm{Ext}^1_S(\mathbb{G}_a,\mathbb{Z}/p \mathbb{Z)}$$ is surjective. Furthermore the morphism $\theta_1$ is an isomorphism.
    \end{thm}
Note that an endomorphism of the additive group $\mathbb{G}_a$ over a perfect field $k$ of characteristic $p$, is equivalent to a ring map $\varphi:k[X]\rightarrow k[X]$, which is completely determined by the element $\varphi(X)$, which corresponds to a polynomial in $f(X) \in k[X]$.
\begin{lem}
    The endomorphisms of the additive group $\mathbb{G}_a$ over a perfect field of characteristic $p$ are in bijection with  polynomials $f(X)$ of the form $f(X)=a_0X+a_1X^p+a_2X^{p^2}+\dots+a_mX^{p^m}$.
\end{lem}
\begin{proof}
    We note that for $f(X)$ to induce a group scheme morphism it has to satisfy the compatibility conditions with the co-multiplication map. It is therefore forced to be additive, i.e $f(X+Y)=f(X)+f(Y)$. Differentiating this with respect to $Y$ and then setting $Y=0$, gives us the identity $f'(X)=f'(0)$. Setting $f(X)=\sum_{i=0}^dc_iX^i$, we get $\sum_{i=0}^dic_iX^{i-1}=c_1$. Therefore $c_i=0$ when $i$ is not divisible by $p$ and greater than $1$. Using this, along with the fact that $f(0)=0$ we get that $f(X)=c_1X+g(X^p)$ for some $g$.\\
    
    We may assume $g(X)\neq 0$. Note furthermore that $f(0)=0$ implies that $g(0)=0$ as well. We have,
    \begin{align*}
        g(X^p+Y^p)&=g((X+Y)^p)\\
        &=f(X+Y)-c_1(X+Y)\\
        &=f(X)+f(Y)-c_1(X)-c_1(Y)\\
        &=g(X^p)+g(Y^p)
    \end{align*}

    This implies that $g$ is additive with respect to the variable $X^p$, which by the previous part allows us to continue the argument to prove the claim.
    \end{proof}
    \begin{proof}
        We prove only the second statement assuming the first. If we let $S=\Spec(A)$, note that for all $a\in A$ we have an equality of morphisms between $a \circ F$ and $F\circ F^{-1}(a)$. Here we think of $a$ and $F^{-1}(a)$ acting via translations. \\

        Note that $F^*(\epsilon_{1,S})=\epsilon_{1,S}$, which implies that $\partial_1(S)(aF)=\theta_1(S)(F^{-1}(a))$. Note by the previous lemma, any endomorphism of $\mathbb{G}_a$ can be written as polynomials the Frobenius, therefore the surjectivity of $\partial_1$  gives also the surjectivity of $\theta_1$.\\

        For injectivity, note that if $\theta_1(S)(a)=0,$ then there exists $u \in \mathrm{End}_S(\mathbb{G}_{a,S})$ such that $a=(F-id)\circ u$. Again by writing $u$ as a polynomial in the powers of the Frobenius we get that $F^r(a)=0$, for $r$ large enough. Therefore, we have that $a=0$. 
    \end{proof}

\begin{thm}
    The morphism $\theta_n$ for $n\geq 1$ is an isomorphism.
\end{thm}
\end{subsubsection}
\begin{subsubsection}{The general case}
    Using the theory in the previous section and various structure theorems for unipotent groups, we prove the following theorem.
    \begin{thm}
        Let $k$ be a perfect field of characteristic $p>0$. For all perfect commutative unipotent connected algebraic groups $G$, the functor $G^*$ is representable, and moreover, $G$ and $G^*$ are isogenous. 
    \end{thm}
\begin{proof}
    We have that all connected commutative connected unipotent groups are a quotient of a finite product of Witt vectors of finite length by a connected subgroup \cite{book:3916}. Therefore there exists an exact sequence,
    $$U_2\xrightarrow{g} U_1\rightarrow G\rightarrow 0$$
    where $U_i=\prod_j W^\mathrm{perf}_{n_{ij}}$. The functor $G^*$ is then the kernel of $g^*$. Through the previous isomorphism in the Witt vector case, we get that $G^*$ is a commutative connected unipotent group scheme.

    On the other hand, we also have that there exists an $k$-isogeny $U\xrightarrow{f} G$ \cite[VIII,~Theorem~1]{serre1960groupes}, where U is a finite product of Witt vectors of finite length. Let $M$ be the finite kernel of $f$, from which we get an exact sequence of $k$-group schemes,
    $$0\rightarrow Hom_k(M,\mathbb{Q}/\mathbb{Z})\rightarrow U^*\xrightarrow{f^*}G^*$$
    if $\overline{k}$ is an algebraic closure of $k$, the group $coker(f^*)(\overline{k})$, which as a subgroup of $\mathrm{Ext}^1_{\overline{k}}(M,\mathbb{Q}/\mathbb{Z})=0$ is also null. Therefore $G^*$ is isogenous to $U^*$ and therefore to $U$ and $G$, which is what we wanted to show. 
\end{proof}
   \begin{thm}\cite[Proposition~1.2.1]{begueri1980dualite}
       The functor $G\rightarrow G^*$ is exact an involutive on the category $\mathfrak{cpu}_k^\circ$.
       \end{thm}
\end{subsubsection}

\begin{subsection}{The non-commutative case}

We detail a suitable duality theory for non-commutative unipotent groups by modifying the commutative case. More specifically we restrict our attention to central extensions, rather than all extensions as in the previous case. Furthermore we give an alternate presentation of this duality theory in the framework of a moduli space of rank 1 local systems.

\begin{defn}[Central extension]
    A central extension of a group scheme $G$ is an isogeny $\widetilde{G}\rightarrow G$ such that the kernel of the morphism is contained in the center of $G$.
\end{defn}
    We modify our definition of the Serre dual of a unipotent commutative group in two ways; first by modifying the kernel of the extensions we are looking at, and second by restricting our attention to only central extensions.
\end{subsection}
\begin{defn}
    If $G$ is a $k$-group scheme, we define a functor $G^*$, which sends a $k$-scheme $S$ to the set of central extensions of $G_S$ by $\mathbb{Q}_p/\mathbb{Z}_p$. 
    
\end{defn}

\begin{prop}\cite[Proposition~A.30]{chsh}
    If $G$ is a connected perfect unipotent group over $k$, the restriction of the functor $G^*$ to $\mathfrak{Perf}_k$
is represented by an object of $\mathfrak{cpu}_k$, which we also denote by $G^*$.
\end{prop}
    We now introduce a different perspective on this duality which will be easier to generalize in later settings. 
\begin{defn}\cite[Definition~1.28]{boyarchenko2013charactersheavesunipotentgroups}
    A multiplicative local system on a perfect connected quasi-algebraic group $H$ over $k$ is a rank 1 $\mathbb{Q}_l$-local system $\mathcal{L}$ equipped with an isomorphism
$$\mu^*(\mathcal{L})\xrightarrow{\sim} pr_1^*(\mathcal{L})\otimes pr_2^*(\mathcal{L})$$

 where $\mu : H \times H \rightarrow H$  denotes the multiplication morphism.
\end{defn}
We can relate these two notions by fixing a homomorphism $\psi:\mathbb{Q}_p/\mathbb{Z}_p \rightarrow \overline{\mathbb{Q}_l^{\times}}$. More precisely if $\widetilde{H}$ is a central extension of $H$ by $\mathbb{Q}_p/\mathbb{Z}_p$, then $\widetilde{H}$ is a $\mathbb{Q}_p/\mathbb{Z}_p$ étale torsor over $H$. The associated rank 1 $\mathbb{Q}_l$-local system induced by the morphism $\psi$, can  be seen to be multiplicative. Furthering these arguments this relation can actually be seen to be a bijection of groups.

\begin{rmk}
    It is important to note that the Serre dual in the context of non-commutative groups is no longer involutive, nor exact in general. Nevertheless, this definition and its associated Fourier transform is useful in a representation theoretic context as outlined in \cite{boyarchenko2013charactersheavesunipotentgroups}.
\end{rmk}
\begin{subsection}{Serre dual of the Perfectization of the  Heisenberg group}
     It will be convenient for us to work in the non-perfect setting first, and then use various adjoint properties of the perfectization functor to transfer our results to the perfect setting.

 \begin{subsubsection}{The Heisenberg group}

\begin{defn}
    Let $W$ be a symplectic vector space over a field $F$, then the Heisenberg group, $H(W)$ is defined to be 
    $$W\oplus F$$
    with group operation
    $$(w_1,t_1) \circ(w_2,t_2)=(w_1+w_2,t_1+t_2+\frac{1}{2}\langle w_1,w_2\rangle)$$
\end{defn}

We can find a Darboux basis $\left\{{e}_j, {f}^k\right\}_{1 \leq j, k \leq n}$ for every symplectic vector space, satisfying $\omega\left({e}_j, {f}^k\right)=\delta_j^k$. Here, $2n$ is the necessarily even dimension of W and $\omega(x,y)$ is the given symplectic form. In this basis, every vector can be written as
$$
v=q^a {e}_a+p_a {f}^a
$$

A vector in $\mathrm{H}(W)$ is then given by
$$
v=q^a {e}_a+p_a {f}^a+t k
$$
where $\{k\}$ is a basis for ${F}$. The group law under this new basis becomes
$$
(p, q, t) \cdot\left(p^{\prime}, q^{\prime}, t^{\prime}\right)=\left(p+p^{\prime}, q+q^{\prime}, t+t^{\prime}+\frac{1}{2}\left(p q^{\prime}-p^{\prime} q\right)\right)
$$
where the multiplication on the right hand side should be understood as the dot product.

Making the substitution $u=t+\frac{1}{2}pq$, we have an isomorphic  matrix presentation where a group element 
$$
v=q^a {e}_a+p_a {f}^a+t k
$$

can be written as 

$$\begin{bmatrix}
    1&p&u\\
    0&I_n&q\\
    0&0&1
\end{bmatrix}$$

where again $p$ and $q$ should be understood in vector notation.

\begin{rmk}
    We can also view the Heisenberg group as a central extension of $W$ by $F$,

    $$0\rightarrow F\rightarrow H(W) \rightarrow W\rightarrow0$$

    Furthermore it is clear to see that the centralizer and commutator subgroup coincide and is isomorphic to $F$. This in particular implies that the abelianization $H^{ab}$ is isomorphic to $W$.
\end{rmk}
\begin{rmk}
    We can transfer these definitions over to the setting of group schemes. A convenient definition that works here is to modify the last remark, so that we end up with an exact sequence,$$0\rightarrow \mathbb{G}_a\rightarrow H(W) \rightarrow \mathbb{G}_a^n\rightarrow0$$
\end{rmk}
\end{subsubsection}  
\begin{subsubsection}{The Set-up}\hspace{5pt}

    We now take a small digression to introduce a few tools that will help us with the calculation.

    \begin{defn}
        Let $G$ be a connected algebraic group over $k$. A group cover of G is a connected algebraic group $\widetilde{G}$ with an isogeny $\widetilde{G}\rightarrow G$. Let \textbf{Cov}(G) be the category of coverings of $G$. One can show that \textbf{Cov}(G) is anti-equivalent to a partially ordered directed poset, where the supremum of two covers is given by the reduced connected component of their fibre product.
    \end{defn}

\begin{defn}
    We let $\Pi_{cent}(G)=\varprojlim A_i$, where the limit is taken over all $A_i$, such that $A_i$ is the kernel of a central group cover in \textbf{Cov}(G).

\end{defn}

Let $A$ be a finite commutative group scheme over $k$. Let 
$$0\rightarrow A\rightarrow \widetilde{G}\rightarrow G\rightarrow0$$ be a central extension. Then $(\widetilde{G})_{red}^\circ$ is a central cover. We construct a map $f_{\widetilde{G}}:\Pi_{cent}(G)\rightarrow A$ in the following way,
$$\Pi_{cent}(G)\twoheadrightarrow \operatorname{ker((\widetilde{G})_{red}^\circ\rightarrow G)\hookrightarrow A}.$$

\begin{lem}\label{3.19}
    \cite{kamgarpour2008stackyabelianizationalgebraicgroup} The mapping $\widetilde{G}\rightarrow f_{\widetilde{G}}$ defines an isomorphism between $H^2(G,A)$; the group of central extensions with kernel $A$, and $\operatorname{Hom}(\Pi_{cent}(G),A)$.
\end{lem}
\begin{proof}
    If we let $B:=\operatorname{im}(\Pi_{cent}(G))$ for some $\varphi \in \operatorname{Hom}(\Pi_{cent}(G), A)$, we must have that $B$ is a finite quotient of $\Pi_{cent}(G)$. Therefore $B \cong A_i$ for some $i$ in the indexing category. Let $G^{\varphi}$ be the pushforward of the central extension
$$
1 \rightarrow A_i \rightarrow G_i \rightarrow G \rightarrow 1
$$
along the morphism $A_i \xrightarrow{\simeq} B \hookrightarrow A$. One checks that $\varphi \mapsto G^{\varphi}$ is the inverse of $\tilde{G} \rightarrow f_{\tilde{G}}$.
\end{proof}

\begin{defn}\label{3.20}
By restricting group covers from $G$ to $[G,G]$, we get a morphism, $\Pi_{cent}([G,G])\rightarrow \Pi_{cent}(G)$. We set $\Pi:\operatorname{Im}(\Pi_{cent}([G,G])\rightarrow \Pi_{cent}(G))$
\end{defn}
Consider the exact sequence
$$0\rightarrow[G,G]\rightarrow G\rightarrow G^{ab}\rightarrow0$$
from which we get an exact sequence 
$$0\rightarrow (G^{ab})^*\rightarrow G^*\rightarrow \operatorname{Im}((G^{ab})^*\rightarrow G^*)\rightarrow0$$

Now notice that $\mathbb{Q}_p/\mathbb{Z}_p$ can be written as a colimit of $\mathbb{Z}/p^n\mathbb{Z}$, which implies that $\operatorname{Hom}(\Pi_{cent}(G),\mathbb{Q}_p/\mathbb{Z}_p)$ is a subset of the product $$\prod_n\operatorname{Hom}(\Pi_{cent}(G),\mathbb{Z}/p^n\mathbb{Z})$$
that satisfy the appropriate compatibilities. This allows us to only consider finite groups from now on. 

In the language of \ref{3.19} the last term is a subset of the product of terms of the form
\begin{align}
    \operatorname{Im}(\operatorname{Hom}(\Pi_{cent}(G),\mathbb{Z}/p^n\mathbb{Z})\rightarrow \operatorname{Hom}(\Pi_{cent}([G,G]),\mathbb{Z}/p^n\mathbb{Z}))\label{1}
\end{align}
Now using the definition from \ref{3.20} we see that this is nothing but $\operatorname{Hom}(\Pi,\mathbb{Z}/p^n\mathbb{Z})$ and the compatibilities between the maps mean that they assemble to form
$$\operatorname{Hom}(\Pi,\mathbb{Q}_p/\mathbb{Z}_p).$$ Therefore we have a exact sequence of the form,
$$0\rightarrow (G^{ab})^*\rightarrow G^*\rightarrow \operatorname{Hom}(\Pi,\mathbb{Q}_p/\mathbb{Z}_p)\rightarrow0$$
We claim now that the group $\Pi$ is trivial in the case of the Heisenberg group scheme, which by further formal arguments will also show that $\Pi$ is also trivial in the case of its perfectization.

\end{subsubsection}
\begin{subsubsection}{Calculation of $\Pi$ for the Heisenberg group}
    \begin{defn}
        Given a central cover $\widetilde{G}\rightarrow G,$ recall that the restriction of $\widetilde{G}$ to $[G,G]$ is given by the scheme $(\widetilde{G}\times_G[G,G])^\circ_{red}$. Denote by $d(\widetilde{G})$ the degree of the associated morphism to $[G,G]$.
    \end{defn}
    \begin{thm}\label{3.22}
        There exists a constant $C$, depending only on $G$, such that $d(\widetilde{G})< C$ for all central covers $\widetilde{G} \rightarrow G$.
    \end{thm}

        In particular this means that $\Pi$ as the image of the map $\Pi_{cent}([G,G])\rightarrow\Pi_{cent}(G)$ is finite.\\

    In order to prove Theorem \ref{3.22}, we use a small lemma. For every positive integer $n$, let $c^n: G^{2n} \rightarrow [G,G] $ denote the map
$$(g_1, g_2, g_3, g_4, ..., g_{2n-1}, g_{2n}) \rightarrow [g_1, g_2][g_3, g_4]...[g_{2n-1}, g_{2n}].$$
\begin{lem}
    There exists a positive integer $n$ such that $c^n$ is surjective.
\end{lem}
    \begin{proof}
        See \cite{milne2006algebraic} Cor. $11.33$
    \end{proof}

    Consider now the following diagram, where $\widetilde{G}$ is a central extension and $H$ is the corresponding restriction to $[G,G]$. 
\[\begin{tikzcd}
	{\widetilde{G}^n} & {G^n} \\
	&& H & {[G,G]} \\
	&& {\widetilde{G}} & G
	\arrow[from=1-1, to=1-2]
	\arrow["{\tilde{c}^n}"', curve={height=12pt}, from=1-1, to=3-3]
	\arrow[dashed, from=1-2, to=2-3]
	\arrow["{c^n}", curve={height=-12pt}, from=1-2, to=2-4]
	\arrow[curve={height=12pt}, squiggly, from=1-2, to=3-3]
	\arrow[from=2-3, to=2-4]
	\arrow[from=2-3, to=3-3]
	\arrow[from=2-4, to=3-4]
	\arrow[from=3-3, to=3-4]
\end{tikzcd}\]
Here $\tilde{c}^n$ is the obvious analogue to the map $c^n$ for the group $\widetilde{G}$. If denote by $A$ the kernel group scheme of the central cover $\widetilde{G}\rightarrow G$, then since the commutator map is invariant under multiplication by elements in A, we get that the map $\tilde{c}^n$ descends to the squiggly map on the diagram. This in turn induces a map to $H$. From this we get a factorization. 
\[\begin{tikzcd}
	{G^{2n}} \\
	H && {[G,G]}
	\arrow[from=1-1, to=2-1]
	\arrow["{c^n}", curve={height=-6pt}, from=1-1, to=2-3]
	\arrow["{degree=d(\widetilde{G})}"', from=2-1, to=2-3]
\end{tikzcd}\]

Due to the fact that $c^n$ is surjective, we can find a closed subvariety $X\subseteq G^{2n}$ such that the map $c^n$ restricted to X is generically finite. Therefore, we see that the number $d(\widetilde{G})$ divides the degree of $c^n|_X$, and hence is bounded.

Let us now apply this to the situation of the Heisenberg group. Let $A_\lambda ,B$ and $C_\lambda$ respectively be the matrices,

\[A_\lambda=
 \begin{bmatrix}
    1&\lambda&0\\
    0&I_n&0\\
    0&0&1
\end{bmatrix},
B=
\begin{bmatrix}
    1&0&0\\
    0&I_n&1\\
    0&0&1
\end{bmatrix},
C_\lambda=
\begin{bmatrix}
    1&0&\lambda\\
    0&I_n&0\\
    0&0&1
\end{bmatrix}
\]
\\
Then we see that $[A_\lambda,B]=C_\lambda$, which means that we only require two copies of the Heisenberg group scheme to form a surjection onto the commutator subgroup scheme. Furthermore, we have the following commutative diagram. Here $H$ is the Heisenberg group.
\[\begin{tikzcd}
	{\mathbb{G}_a} && {H^2} \\
	&& {\mathbb{G}_a}
	\arrow["{\lambda\rightarrow(A_\lambda,B)}", from=1-1, to=1-3]
	\arrow[from=1-1, to=2-3]
	\arrow[from=1-3, to=2-3]
\end{tikzcd}\]

Since the diagonal map is an isomorphism we can use the cancellative property of closed immersions to show that the top map is a closed immersion. This shows that the number $d(\widetilde{G})=1$, therefore proving that $\Pi$ is trivial in the case of the Heisenberg group scheme.

\end{subsubsection}

\begin{subsubsection}{Transferring ideas to the perfect case }
    Recall that the perfectization functor was defined as an inverse limit in the category of schemes, but when viewed in the category of rings it becomes a direct limit. Due to the fact that direct limits preserve surjections of rings, the perfectization functor in the category of schemes preserves closed immersions.\\

    In particular, it preserves the closed immersion of $\mathbb{G}_a\hookrightarrow H^2$, constructed previously. The above discussion also gives us that the commutator subgroup scheme of $H^\mathrm{perf}$ is $\mathbb{G}_a^\mathrm{perf}$, and also that the abelianization of $H^\mathrm{perf}$ is $(\mathbb{G}_a^n)^\mathrm{perf}$.\\

    The results from the previous section also prove that $\Pi$, in the case of the perfectization of the Heisenberg group scheme, is also trivial. This along with the exact sequence constructed previously
    $$0\rightarrow (H^{ab})^*\rightarrow H^*\rightarrow \operatorname{Hom}(\Pi,\mathbb{Q}_p/\mathbb{Z}_p)\rightarrow0$$
    and theorem \ref{3.5} shows that the Serre dual of the perfectization of the Heisenberg group is isomorphic to $(\mathbb{G}_a^n)^\mathrm{perf}$.

\end{subsubsection}
\end{subsection}

\end{subsection}

\section{The Fargues-Fontaine curve and vector bundles}
In this section, we aim to introduce the Fargues-Fontaine curve, and explore some properties that relate it to more understandable objects. We also detail the associated diamond, which is closely tied to the notion of a Cartier divisor on the curve. Finally, we introduce the notion of vector bundles on the curve and state a few results regarding their cohomologies. We follow the exposition in \cite{fargues2024geometrizationlocallanglandscorrespondence} and \cite{scholze2020berkeley}.
\begin{subsection}{The construction of the Fargues-Fontaine Curve}
As usual, we fix a local non-Archimedean field $E$, which can either be a finite extension of $\mathbb{Q}_p$ or $\mathbb{F}_q((\pi))$. These two cases are known as the \textit{mixed} and \textit{equal characteristic} cases respectively. We fix a uniformizer $\pi$, generating the maximal ideal $\mathfrak{m}$ of $\mathcal{O}_E$; the ring of integers of $E$. We call $q=p^n$ the cardinality of the residue field $\mathfrak{e}=\mathcal{O}_E/\pi$.\\

The construction is motivated by the need to generalize the theory of shtukas, where one has to make sense of what the product "$S\times\operatorname{Spa}(\mathbf Z_p)$" should be. The correct candidate, it turns out, is modeled after the fact that for any perfect $\mathbb{F}_p$-algebra, there exists a unique $\pi$-adically complete, flat $\mathcal{O}_E$-algebra $\widetilde{R}$ such that $\widetilde{R}/\pi=R$.\\

We apply this to the situation where $S=\operatorname{Spa}(R,R^+)\in \operatorname{Perf}_\mathfrak{e}$ is a perfectoid space over $\mathfrak{e}$. Here $R^+$ is perfect, as it is an integrally closed subring within a perfect ring. The fact that $R$ is perfect is the content of \cite[Proposition~6.1.6]{scholze2020berkeley}. If $E$ is of mixed characteristic then this is the $\mathcal{O}_E$-algebra of Witt vectors $W_{\mathcal{O}_E}(R^+)$, and in the equal characteristic case it is $R^+[[\pi]]$.
Let $\varpi$ be a pseudouniformizer in $R$, and let $[-]:R^+\rightarrow W_{\mathcal{O}_E}(R^+)$ denote the Teichmüller lift. We set
$$\mathcal{Y}_S=\operatorname{Spa}(W_{\mathcal{O}_E}(R^+),\operatorname{Spa}(W_{\mathcal{O}_E}(R^+))\backslash V([\varpi])$$

Here, we endow $W_{\mathcal{O}_E}(R^+)$ with the $(\pi,[\varpi])$-adic topology. \\

We can construct a Frobenius morphism on the ring $W_{\mathcal{O}_E}(R^+)$, by lifting the Frobenius on $R^+$. This then descends to an automorphism of $\mathcal{Y}_S$.
\\

\begin{prop}\label{5.1}\cite[Propostion~II.1.1]{fargues2024geometrizationlocallanglandscorrespondence}
The above defines an analytic adic space $\mathcal{Y}_S$ over $\mathcal{O}_E$. Letting $E_{\infty}$ be the completion of $E(\pi^{1/p^\infty})$, the base change
$$\mathcal{Y}_S\times_{\operatorname{Spa}\mathcal{O}_E}\operatorname{Spa}\mathcal{O}_{E_{\infty}}$$
is a perfectoid space, with tilt given by
$$S\times_{\mathbb{F}_q}\Spa\mathbb{F}_q[[t^{1/p^\infty}]]=\mathbb{D}_{S,\Perf}$$
A perfectoid open disc over S.
\end{prop}
\begin{proof}
    We will first consider the case where $S=\Spa(R,R^+)$. Recall that a point $x\in \Spa(R,R^+)$ is defined to be non-analytic if its kernel is an open subset in $R$.  However, since the vanishing locus of $[\varpi]$ has been excluded, we see that no power of $[\varpi]$ can be contained in the kernel which means that it cannot be open. This is an equivalent definition for a space to be analytic. \\

    We have an open covering of $\mathcal{Y}_S$ by rational subsets defined by the inequality $|\pi|\leq|[\varpi]^{1/p^n}|\neq0$ for $n=1,2\dots$.Let $\Spa(R_n,R_n^+)$ be the rational subset associated to the inequality $|\pi|\leq|[\varpi]^{1/p^n}|\neq0$. The ring $R_n$ is obtained by $[\varpi]$-adically completing $W_{\mathcal{O}_E}(R^+)\left[\frac{\pi}{[\varpi]^{1/p^n}}\right]$ and then inverting $[\varpi]$. We can restrict to completing by $\varpi$ since $\pi^p\subset([\varpi])$ in this ring.\\

    We use a small lemma to give an alternate presentation of the ring $R_n$. Recall that the Tate algebra of a topological ring $A$ is defined to be
    $$A\langle T\rangle=\bigg\{f=\sum a_iT^i\mid a_i\rightarrow 0,i\rightarrow\infty\bigg\}$$

    It is well known that the above then simplifies to $$W_{\mathcal{O}_E}(R^+)\Big\langle\frac{\pi}{[\varpi]^{1/p^n}}\Big\rangle\left[\frac{1}{[\varpi]}\right]$$

    Note also that we can write every element of $W_{\mathcal{O}_E}(A)$ for a complete $\mathbb{F}_q$-algebra $A$ as a Laurent series using the Teichmuller lifts. 

    $$(a_i)\in W_{\mathcal{O}_E}(A)=\sum_{i\geq N} [{a_i}] \pi^i$$

   Expanding on this presentation before inverting $[\varpi]$, we get

   $$\sum_{j,k}[r_{j,k}]\pi^k\bigg(\frac{\pi}{[\varpi^{1/p^n}]}\bigg)^j$$ Such that
   
   $$\sum_{j,k}[r_{j,k}]\pi^k\rightarrow 0 \hspace{5pt}\text{as}\hspace{5pt} j\rightarrow0$$

   We can rewrite this as
\begin{align*}
    \sum_{j,k}[r_{j,k}][\varpi]^{k/p^n}\frac{\pi^k}{[\varpi]^{k/p^n}}\bigg(\frac{\pi}{[\varpi^{1/p^n}]}\bigg)^j\\
    \sum_{j,k}[r_{j,k}][\varpi]^{k/p^n}\bigg(\frac{\pi}{[\varpi^{1/p^n}]}\bigg)^{j+k}\\
    \sum_i\sum_{j+k=i}[r_{j,k}][\varpi]^{k/p^n}\bigg(\frac{\pi}{[\varpi^{1/p^n}]}\bigg)^{i}
\end{align*}
We now claim that $\sum_{j+k=i}[r_{j,k}][\varpi]^{k/p^n}$ tends to zero as $i$ goes to infinity. Let us first explicitly describe the $[\varpi]$-adic topology on $W_{\mathcal{O}_E}(R^+)$. The ideals of generation in this topology are given by $[\varpi]^n \cdot W_{\mathcal{O}_E}(R^+)$, whose elements can be explicitly stated as $(\varpi^n,0,0,0\dots)\cdot(a_0,a_1,a_2,\dots)=(\varpi^na_0,\varpi^{n^q}a_1,\varpi^{n^{q^2}}a_2).$ This implies that since $\sum_k[r_{j,k}]\pi^k$ tended to zero, choosing a fixed $k$, $r_{j,k}\rightarrow0$ as $j\rightarrow\infty$.\\

Consider any valuation $|-|$ on $\Spa(R_n,R_n^+)$ and let $\epsilon\geq0$, by the previous argument we can chose $j_0$ such that $\big|[r_{j,k}] \pi^{k/p^n}\big|\leq\epsilon/2$ whenever $j\geq j_0$. Similarly, by the fact that $[\varpi]$ is topologically nilpotent we choose $k_0$ such that $\big|[r_{j,k}][\varpi]^{k/p^n}\big|\leq\epsilon/2$ for $k\geq k_0$. This implies

\begin{align*}
    &\bigg|\sum_{j+k=i}[r_{j,k}][\varpi]^{k/p^n}\bigg|\\
    \leq&\max_{j+k=i}\bigg|\sum_{j+k=i}[r_{j,k}][\varpi]^{k/p^n}\bigg|
\end{align*}

For $i\geq j_0+k_0$, we have 
\begin{align*}
    &\max_{j+k=i}\bigg|\sum_{j+k=i}[r_{j,k}][\varpi]^{k/p^n}\bigg|\leq\epsilon
\end{align*}
this gives us the claim. Now consider what happens when we invert the element $[\varpi]$. In a similar way to how $\mathbb{Z}_p\langle T \rangle\big[\frac{1}{p}\big]=\mathbb{Q}_p\langle T \rangle$, we have the following description for the ring $R_n$
\begin{align}\label{1.00}
    R_n=\bigg\{\sum_i[r_i]\bigg(\frac{\pi}{[\varpi^{1/p^n}]}\bigg)^i\Bigg|\hspace{5pt}r_i\in R,r_i\rightarrow0\bigg\}
\end{align}

Since the Frobenius automorphism sends $R_{n-1}$ to $R_n$, to see that $\mathcal{Y_S}$ is an adic space, it suffices to check that the structure sheaf on $R_0$ is sheafy. To this end, we use the theory of sousperfectoid spaces in \cite[6.3]{scholze2020berkeley}. Since the canonical injection $R_0\hookrightarrow R_0\widehat{\otimes}_{\mathcal{O}_E}\mathcal{O}_{E_\infty}$ admits a section, it suffices to prove that $R_0\widehat{\otimes}_{\mathcal{O}_E}\mathcal{O}_{E_\infty}$  is perfectoid.\\

Let us call, $$A=R_0\widehat{\otimes}_{\mathcal{O}_E}\mathcal{O}_{E_\infty}$$ and $A^+\subset A$ the integral closure of $R_0^+\widehat{\otimes}_{\mathcal{O}_E}\mathcal{O}_{E_\infty}$. Consider now the ring $$A_0^+=(W_{\mathcal{O}_E}(R^+)\widehat{\otimes}_{\mathcal{O}_E}\mathcal{O}_{E_\infty})\Big[\Big(\frac{\pi}{[\varpi]}\Big)^{\frac{1}{p^\infty}}\Big]^{\wedge}_{[\varpi]}\subset A^+$$

Let us now recall from \cite[Lemma~3.10(ii)]{bhatt2018integral}, that a ring $S$ is called \textit{integral perfectoid} if it is $t$-adically complete with respect to some element $t\in S$ 
such that $t^p$ divides $p$, and if the Frobenius endomorphism $ S/ t S \rightarrow S/t^pS$ is an isomorphism. This is an important notion, as it is proven in the same that if $S$ is integral perfectoid, then $S[\frac{1}{t}]$ is perfectoid in our sense.\\

Since $\pi\subset([\varpi])$, when considering the quotient $A_0^+/[\varpi]$, we can first quotient out by $\pi$, and then further by $[\varpi]$. Combining this with the fact that we have distinguished elements $[\varpi]^{1/p^n}$, and that $ W_{\mathcal{O}_E}(R^+)/\pi =R^+$, implies that we have the following description for $A_0^+/[\varpi]$,

$$A_0^+/[\varpi]=((R^+/\varpi)\widehat{\otimes}_{\mathbb{F}_q}\mathbb{F}_q[\pi^{1/p^\infty}])[t^{1/p^\infty}]/(\pi^{1/p^n}-[\varpi]^{1/p^n}t^{1/p^n},\pi)=(R^+/\varpi)[t^{1/p^\infty}]/(t)$$.

This description clearly shows that the Frobenius on this ring is an isomorphism because it was already assumed to be such on $R^+/\varpi$. To complete the first part of the proof, we observe that the $p^{th}$ power of $[\varpi]^{1/p}$ divides $\pi$, and as a consequence of the previous arguments we have that $A_0^+/[\varpi]^{1/p}\rightarrow A_0^+/[\varpi]$ is also an isomorphism and moreover $A_0^+\big[\frac{1}{[\varpi]^{1/p}}\big]=A$. Thus we have satisfied all the hypothesis required to use \cite[Lemma~3.10(ii)]{bhatt2018integral}.\\

We now consider the tilt of $A$. Note that $A^\flat=(A^+_0)^\flat[\frac{1}{[\varpi]}]$. Notice furthermore that

$$A^+_0/\pi=R^+\otimes_{\mathbb{F}_q}\mathbb{F}_q[t^{1/p^\infty}]/(t)$$
which implies that
$$(A^+_0)^\flat=\varprojlim (A^+_0/\pi)=\varprojlim \Big( R^+\otimes_{\mathbb{F}_q}\mathbb{F}_q[t^{1/p^\infty}]/(t)\Big)=\varprojlim \Big( R^+[t^{1/p^\infty}]/(t)\Big)$$
From \cite[Example~2.0.3]{bhatt2017lecture} we have that if $R$ is a prefect ring of characterstic $p$, and $f\in R$ is a nonzerodivisor then the inverse limit along the frobenious map of $(R/f)$ is the $f$-adic completion of $R$. This gives us that
$$\varprojlim \Big( R^+[t^{1/p^\infty}]/(t)\Big)=R^+[t^{1/p^\infty}]^\wedge_t=R^+\langle t^{1/p^\infty}\rangle$$
which upon inverting $[\varpi]$ implies that
\begin{align}\label{4}
    A^\flat=R\langle t^{1/p^\infty}\rangle.
\end{align}
Now consider the perfection of the open unit disc over $S$. This is given as the increasing union, $\mathbb{D}_{S,\text{Perf}}=\bigcup_{n\geq1}\Spa R\langle T,T/\varpi^n\rangle_{\text{perf}}$, for which we have the rational open subset corresponding to $|T|\leq|\varpi|$ is,
$$\Spa R\langle T,T/\varpi\rangle_{\text{perf}}=\Spa R\langle T/\varpi\rangle_{\text{perf}}=\Spa R\langle (T/\varpi)^{1/p^\infty}\rangle$$

We therefore have an abstract isomorphism with the presentation in (\ref{4}), and moreover the various $R_n$ as in (\ref{1.00}) give a full cover of the open unit disc, corresponding to the other rational open subsets.
\end{proof}

\begin{defn}[étale morphisms of perfectoid spaces. {\cite[Definition~7.5.1]{scholze2020berkeley}}]

\hspace{20pt}\begin{enumerate}
      \item A morphism $f:X\rightarrow Y$ of perfectoid spaces is called \textit{finite étale} if for all open affinoids $\Spa(B,B^+)\subset Y$, the pullback $X\times_Y\Spa(B,B^+)$ is an affinoid space $\Spa(A,A^+)$, where $A$ is a finite étale $B$-algebra, and $A^+$ is the integral closure of $B^+$ in $A$.
      \item A morphism $f:X\rightarrow Y$ of perfectoid spaces is called \textit{étale} if for all $x\in X$, there exists and open set $U\ni x$ and $V\subset f(U)$, such that there is a diagram
\[\begin{tikzcd}
	U && W \\
	& V
	\arrow["open", hook, from=1-1, to=1-3]
	\arrow["{f|_U}"', from=1-1, to=2-2]
	\arrow["{finite\ étale}"{pos=0}, from=1-3, to=2-2]
\end{tikzcd}\]
  \end{enumerate}
\end{defn}

\begin{defn}[Pro-étale morphisms. {\cite[Definition~8.2.1]{scholze2020berkeley}}]

A morphism of affinoid perfectoid spaces $f:\Spa(B,B^+)\rightarrow\Spa(A,A^+)$ is called \textit{affinoid pro-étale} if 
$$(B,B^+)=\widehat{\varinjlim(A_i,A_i^+)}$$
is a completed filtered colimit of pairs $(A_i,A_i^+)$ such that $A_i$ is perfectoid and,
$$\Spa(A,A^+)\rightarrow \Spa(A_i,A_i^+)$$ 
is étale. A morphism of perfectoid spaces is called \textit{pro-étale} if it both locally on the source and target affinoid pro étale.

\end{defn}
\begin{defn}[The big pro-étale site. {\cite[Definition~8.2.6]{scholze2020berkeley}}]
We endow the following categories,
\begin{itemize}
    \item Perfd, the category of all perfectoid spaces
    \item $\text{Perf}\subset\text{Perfd}$, the category of perfectoid spaces of characteristic $p$.
\end{itemize}
with a Grothendieck topology, where a collection of morphisms $f_i:Y_i\rightarrow Y$ is called a cover if all the $f_i$ are pro-étale, and furthermore for any quasi-compact open $U\subset Y$, there exists a finite subset $I_U\subset I$, and $U_i\subset Y_i$ for $i\in I_U$, such that $U$ is the union of all the $f_i(U_i)$.
    
\end{defn}
\begin{prop}[{\cite[Corollary 8.6]{etcoh}}]
    Let $X$ be a perfectoid space. The presheaf on the pro-étale site Perfd, represented by $X$ is a sheaf. 
\end{prop}

\begin{lem}
    The absolute product of two perfectoid spaces of characteristic $p$ is again a perfectoid space of characteristic $p$.
\end{lem}
For the next definition, we also denote by $X$, the sheaf represented by $X$. 

\begin{defn}[{\cite[Definition~8.3.1]{scholze2020berkeley}}]
    A \textit{diamond} is a pro étale sheaf on Perf, that has a presentation of the form $X/R$, where $X$ a perfectoid space of characteristic $p$, and $R\hookrightarrow X \times X$ is an equivalence relation that is also a perfectoid space such that the two projections $p_1,p_2:R\rightarrow X$ are pro-étale.
\end{defn}
\begin{defn}[{\cite[Definition~10.1.1]{scholze2020berkeley}}]
    Let $X$ be an analytic pre-adic space over $\Spa \mathbf{Z}_p$. We define a presheaf $X^\diamondsuit$ on Perf as follows. For  $T\in $Perf, let $X^\diamondsuit(T)$ be the set of isomorphism classes of pairs $(T^{\sharp},T^{\sharp}\rightarrow X)$, where $T^{\sharp}$ is an untilt of $T$, and $T^{\sharp}\rightarrow X$ is a map of pre-adic spaces. We denote by $\mathrm{Spd}(R,R^+)$ the diamond associated to $\Spa(R,R^+)$.
\end{defn}
\begin{prop}[{\cite[Lemma~15.6]{etcoh}}] The presheaf $X^\diamondsuit$ is a diamond.
\end{prop}
\begin{lem}
    Let $\widehat{\text{Alg}}_{\mathcal{O}_E}$ denote the category of $\pi$-adically complete, $\pi$-torsion free, $\mathcal{O}_E$-algebras and $\text{Perf}_{\mathbb{F}_p}$ denote the category of perfect $\mathbb{F}_p$-algebras. The functor 
$$(-)^\flat:\widehat{\text{Alg}}_{\mathcal{O}_E}\rightarrow \text{Perf}_{\mathbb{F}_p} $$
admits a left adjoint given by $W_{\mathcal{O}_E}(-)$.
\end{lem}
\begin{prop}[{\cite[II.1.2]{fargues2024geometrizationlocallanglandscorrespondence}}\label{5.11}] We have a natural isomorphism,
$$\mathcal{Y}_S^\diamondsuit\cong\mathrm{Spd}\hspace{1pt}\mathcal{O}_E\times S$$
\end{prop}

\begin{proof}
Since both sides are sheaves for the pro-étale topology, we may assume that $T=\Spa(A,A^+)\in$ Perf.
A morphism of adic spaces $T^\sharp\rightarrow\mathcal{Y}_S$ is equivalent to a morphism of rings $W_{\mathcal{O}_E}(R+)\rightarrow R^{\sharp +}$ such that the image of $[\varpi]$ is invertible in $A^\sharp$. Note that by lemma \ref{5.11} we have,
\begin{align*}
    \Hom(W_{\mathcal{O}_E}(R^+),A^{\sharp +})
    \cong\Hom(R^+,(A^{\sharp +})^\flat)
\end{align*}

Furthermore, we have isomorphisms $(A^{\sharp +})^\flat\cong (A^{\sharp \flat})^+\cong A^{+}$ \cite[Proposition~9.1.2]{bhatt2017lecture}. Note moreover, that since the adjunction is given by precomposition with the unit,
$$R^+\cong \varprojlim R^+\cong \varprojlim \Big(W_{\mathcal{O}_E}(R^+)/\pi\Big) \cong \Big(W_{\mathcal{O}_E}(R^+)\Big)^\flat $$
we have the following commutative diagram of multiplicative monoids,
\[\begin{tikzcd}
	{W_{\mathcal{O}_E}(R^+)} & {W_{\mathcal{O}_E}(R^+)[\frac{1}{[\varpi]}]} & {A^{\sharp +}[\frac{1}{t^\sharp}]} \\
	{R^+} & {R^+[\frac{1}{\varpi}]} & {A^+[\frac{1}{t}]}
	\arrow[from=1-1, to=1-2]
	\arrow[from=1-2, to=1-3]
	\arrow[from=2-1, to=1-1]
	\arrow[from=2-1, to=2-2]
	\arrow[two heads, from=2-2, to=1-2]
	\arrow[from=2-2, to=2-3]
	\arrow[two heads, from=2-3, to=1-3]
\end{tikzcd}\]

from which we see that the right vertical, map which is known as the sharp map, is surjective. Since the image of $[\varpi]$ was assumed to be invertible in ${A^{\sharp +}[\frac{1}{t}]}$, we have that the image of $\varpi$ is also invertible in ${A^+[\frac{1}{t}]}$.
\end{proof}

\begin{defn}[{\cite[Definition~10.3.1]{scholze2020berkeley}}]
    If $\mathcal{D}$ is a diamond, with presentation $\mathcal{D}=X/R$, then we set the underlying topological space of $\mathcal{D}$ to be the quotient $|X| / |R|$.
\end{defn}
\begin{rmk}
    The underlying topological space of a diamond is independent of the chosen presentation. \cite[Proposition~11.13]{etcoh}.
\end{rmk}

\begin{prop}[{\cite[Proposition~10.3.7]{scholze2020berkeley}}]
Let $X^\diamondsuit$ be the diamond associated to an analytic adic space $X$. Then we have that the underlying topological spaces of both the diamond and the analytic adic space are naturally isomorphic.
    
\end{prop}

We therefore have a natural projection map $w_S:|\mathcal{Y}_S|\approx|\mathcal{Y}_S^\diamondsuit|=|S\times\mathrm{Spd} \hspace{1pt}\mathcal{O}_E|\rightarrow |S|$, that is functorial in $S$.
\begin{defn}
    We define the \textit{v-topology} on the category Perfd of all perfectoid spaces, where $\{f_i:Y_i\rightarrow Y\}_{i\in I}$ is a cover if and only if for all quasi compact open subsets $V\subset Y$ we can find a finite subset $I_V\subset I$ and quasi compact opens $V_i\subset Y_i$ for $i\in I_V$ such that $V=\bigcup_{i\in I_V}f(V_i)$.
 \end{defn}

\begin{rmk}
     Note that since maps of affinoid analytic spaces are spectral, i.e quasi compact, any surjection $X\rightarrow Y$ of affinoid perfectoid spaces constitutes a v-cover.
\end{rmk}
\begin{lem}\label{5.17}
    If $f:X\rightarrow Y$ is a surjective map between affinoid adic spaces, then it is a quotient map in the topological sense.
\end{lem}
\begin{prop}[{\cite[Proposition~II.1.3 \label{5.18}]{scholze2020berkeley}}] Let $S'\subset S$ be an affinoid open subset, then we have that the natural map $\mathcal{Y}_{S'}\rightarrow\mathcal{Y}_S$ is an open immersion of analytic adic spaces and the following diagram is cartesian,
\[\begin{tikzcd}
	{|\mathcal{Y}_{S'}|} & {|\mathcal{Y}_S|} \\
	{|S'|} & {|S|}
	\arrow[from=1-1, to=1-2]
	\arrow["{w_{S'}}"', from=1-1, to=2-1]
	\arrow["\lrcorner"{anchor=center, pos=0.125}, draw=none, from=1-1, to=2-2]
	\arrow["{w_S}", from=1-2, to=2-2]
	\arrow[from=2-1, to=2-2]
\end{tikzcd}\]
    
\end{prop}

\begin{proof}
    Let $Z$ be the analytic space associated to the open set $|\mathcal{Y}_S|\times_{|S|}|S'|$, then by functoriality we get a map of adic spaces $\mathcal{Y}_{S'}\rightarrow Z$. To prove the proposition, we need to show that this map is an isomorphism. We claim that this can be done after base change to $\mathcal{O}_{E_\infty}$. \\
    
Indeed, as $\Spa \mathcal{O}_{E_\infty}\rightarrow\Spa O_E$ is a surjective map of affinoids it is a v-cover, which by lemma \ref{5.17} is a quotient map on topological spaces. It is well known that the topological notions of surjectivity, injectivity and the property of being an open map can be checked after base change along a surjective quotient map. Together, these mean that the underlying homeomorphism of topological spaces can be checked after the base change. \\

    We set $Z_{\mathcal{O}_{E_\infty}}:=Z\times_{\Spa  \mathcal{O}_E} \Spa \mathcal{O}_{E_\infty}$ and $\mathcal{Y}_{S_{\mathcal{O}_{E_\infty}}}:=\mathcal{Y}_{S}\times_{\Spa  \mathcal{O}_E} \Spa \mathcal{O}_{E_\infty}$ and denote by $p$ the respective projections. We have that the canonical morphism of sheaves $\mathcal{O}_Z\rightarrow p_*\mathcal{O}_{Z_{\mathcal{O}_{E_\infty}}}$ and $\mathcal{O}_{\mathcal{Y}_{S}}\rightarrow p_*\mathcal{O}_{\mathcal{Y}_{S_{\mathcal{O}_{E_\infty}}}}$ are split injective. Which imply that we can check the isomorphism of structure sheaves along the base change as well.\\

    The isomorphism on the level of the base change follows from the tilting equivalence and proposition \ref{5.11} after passing to diamonds.

\end{proof}

We can now generalize the construction of $\mathcal{Y}_S$ to an arbitrary perfectoid space $S\in$ Perf. Given a cover $\{S_i\}$ of $S$, one first constructs $\mathcal{Y}_{S_i}$. By Proposition \ref{5.18} we have that $\mathcal{Y}_{S_i\times_S S_j}\subset\mathcal{Y}_{S_i}$ is an open immersion, and these various inclusions satisfy the cocycle condition by functoriality. Thus we obtain an analytic adic space $\mathcal{Y}_S$ such that $\mathcal{Y}_S^\diamondsuit\cong S\times\mathrm{Spd}\hspace{1pt}\mathcal{O}_E$ and that $\mathcal{Y}_S\times_{\Spa\mathcal{O}_E} \Spa \mathcal{O}_{E_{\infty}}$ is perfectoid.

\begin{subsubsection}{The radius map.}\label{5.11} We briefly return to the case where $S$ is affinoid to introduce the radius map $\kappa:|\mathcal{Y}_S|\rightarrow [0,\infty)$, which measures the relative relationship between the values of the  global section $[\varpi]$ and $\pi$, at each point. In Proposition \ref{5.1} we showed that the space $\mathcal{Y}_S$ was analytic. This implies that each point $x\in\mathcal{Y}_S$ has a unique rank 1 generalization $(\eta_x)$, which we can assume to take values in the real numbers. The map,
\begin{align*}
\kappa:|\mathcal{Y}_S|&\longrightarrow [0,\infty)\subset \mathbb{R}\\
x&\longmapsto\frac{\mathrm{log}|[\varpi](\eta_x)|}{\mathrm{log}|[\pi](\eta_x)|}
\end{align*}
is well defined since equivalent rank 1 valuations have proportional logarithms. Notice that since $\eta_x$ is continuous, and $[\varpi]$ and $\pi$ are topologically nilpotent, we have that $0<|[\varpi](\eta_x)|<1,0\leq|[\pi](\eta_x)|< 1$. which implies that the fraction is always positive and well defined.\\

Furthermore, this map is continuous as the preimage of the bounded open interval $(a,a')\in \mathbb{R}$ with $a=\frac{m}{n},a'=\frac{m'}{n'}\in\mathbb{Q}_{\geq 0}$ is given by the following rational open subset

$$\mathcal{Y}_S \supset \mathcal{Y}_{S,[a,a']}:=\Big\{x\hspace{1pt}|\hspace{1pt}mn'\leq \frac{\mathrm{log}|[\varpi]^{nn'}(\eta_x)|}{\mathrm{log}|[\pi](\eta_x)|}\leq m'n\Big\}=\Big\{x\hspace{1pt}|\hspace{1pt}|[\pi]^{m'n}(\eta_x)|\leq|[\varpi]^{nn'}(\eta_x)| \leq |[\pi]^{mn'}(\eta_x)| \Big\}.$$

\end{subsubsection}
\begin{defn}For an arbitrary $S\in$ Perf, we define the analytic adic space $Y_S:=\mathcal{Y_S}\times_{\Spa\mathcal{O}_E}\Spa E$. 
\end{defn} 

\begin{rmk}
In the affinoid case, note that $\kappa(x)=0$ if and only if $|[\pi](\eta_x)| = 0$, which means that $\eta_x$ factors through $\mathbb{F}_q$, corresponding to the special point in $\Spa \mathcal{O}_E$. This implies the space $Y_S$ has the following equivalent presentations,
$$Y_S=\kappa^{-1}(0,\infty)=\mathcal{Y_S}\times_{\Spa\mathcal{O}_E}\Spa E=W_{\mathcal{O}_E}(R^+) \backslash V(\pi[\varpi])$$
\end{rmk}

\begin{defn}
    For affinoid $S=\Spa(R,R^+)\in $ Perf, the ring $W_{\mathcal{O}_E}(R^+)$ carries a Frobenius morphism
\begin{align*}
    \varphi:W_{\mathcal{O}_E}(R^+)&\xrightarrow{\cong} W_{\mathcal{O}_E}(R^+)\\
    \sum_{n=0}^\infty[a_n]\pi^n&\longmapsto \sum_{n=0}^\infty[a_n^q]\pi^n    
\end{align*}
which is a automorphism since $R^+$ is perfect. Using proposition \ref{5.18}, these morphisms glue to a morphism of $\mathcal{Y}_S$ and $Y_S$ over an arbitrary perfectoid space. 
\end{defn}
\begin{rmk}
    The action of $\varphi$ on $Y_S$ is free and totally discontinuous. To see this we can reduce to the affinoid case, where we have

    $$\kappa(\varphi(x))=\frac{\mathrm{log}|[\varpi^q](\eta_x)|}{\mathrm{log}|[\pi](\eta_x)|}=q\cdot\kappa(x)$$
\end{rmk}

\begin{defn}[{\cite[Definition~II.1.15]{fargues2024geometrizationlocallanglandscorrespondence}}]Given a perfectoid space $S\in$ Perf, we define the Fargues-Fontaine curve to be the quotient $$X_S:=Y_S/\varphi^{\mathbb{Z}}$$
\end{defn}

\begin{prop}[{\cite[Proposition~II.1.17]{fargues2024geometrizationlocallanglandscorrespondence}}]\label{5.24} We have a natural isomorphism of diamonds
$$Y_S^{\diamondsuit} \cong S\times \mathrm{Spd}\hspace{1pt}E $$
which descends to an isomorphism
$$X_S^\diamondsuit\cong S/\varphi_S^\mathbb{Z}\times\mathrm{Spd}\hspace{1pt}E$$
\end{prop}
Notice that the absolute Frobenius morphism on $X_S^\diamondsuit$ given by $\varphi_S\circ\varphi_{\mathrm{Spd}\hspace{1pt}E}$ is the identity on the underlying topological spaces, since it preserves all equivalence classes of valuations. This implies that we have the following chain of maps
$$|X_S|\cong |S/\varphi_S^\mathbb{Z}\times\mathrm{Spd}\hspace{1pt}E|\cong |S\times \mathrm{Spd}\hspace{1pt}E/\varphi_{\mathrm{Spd}\hspace{1pt}E}^\mathbb{Z}| \rightarrow |S|.$$
Moreover, this projection map is even a morphism of v-sites which we call
$$\tau:(X_S)_v\longrightarrow (S)_v$$
\begin{prop}[{\cite[Proposition~II.1.18]{fargues2024geometrizationlocallanglandscorrespondence}}]
Let $S\in$ Perf. The following objects are in a natural bijection
\begin{enumerate}
    \item Sections of $Y_S\rightarrow S$
    \item Maps $S\rightarrow \mathrm{Spd}\hspace{1pt}E$
\item Untilts $S^\sharp$ over $E$ of $S$.
\end{enumerate}
    Consider now the map,
    \begin{align*}
        \theta:W_{\mathcal{O}_E}(R^+)&\rightarrow (R^\sharp)^+\\
        \sum[r_i]\pi^i&\mapsto\sum r_i^\sharp \pi^i
        \end{align*}
Here we use the identification $(R^\sharp)^\flat\cong R$ together with the map $(-)^\sharp:R^\flat\rightarrow R$, which is the canonical projection under the identification of multiplicative monoids $$R^\flat\cong\varprojlim_{x\rightarrow x^p}R.$$
Therefore, given an untilt $S^\sharp$ of S, we have a natural map $S^\sharp\hookrightarrow Y_S$, which is a closed Cartier divisor. Moreover we have that the composite $S^\sharp\hookrightarrow Y_S\rightarrow X_S$ is also a closed Cartier divisor.
\end{prop}
\end{subsection}

\begin{subsection}{Vector bundles on the Curve}\label{5.2}
Fixing an affinoid perfectoid space $S=\Spa(R,R^+)$, we may use the projection map $p:Y_S\rightarrow X_S$ to study vector bundles on $X_S$. Note that we have an open cover of $Y_S$ given by,

$$Y_S=\bigcup Y_{S,[q^{-i},q^i]}$$

By the discussion in section \ref{5.11}, we see that 
$$\mathcal{O}_{Y_S}(Y_{S,[q^{-i},q^i]})=W_{\mathcal{O}_E}(R^+)\Big\langle\frac{\pi^{q^i}}{[\varpi]},\frac{[\varpi^{q^i}]}{\pi}\Big\rangle\Big[\frac{1}{[\varpi]}\Big].$$
Using this presentation we can show that
$$\mathcal{O}_{Y_S}(Y_{S,[q^{-i-1},q^{i+1}]})\rightarrow\mathcal{O}_{Y_S}(Y_{S,[q^{-i},q^i]})$$ 
is a continuous homomorphism of Banach $E$-algebras with dense image.\\

Consider now an vector bundle $\mathcal{E}$ on $X_S$, with pullback $\mathcal{E}|_{Y_S}:=\mathcal{F}$. Since $Y_{S,[q^{-i},q^i]}$ is affinoid we have that the cohomology groups
$$H^i(Y_{S,[q^{-i},q^i]}, \mathcal{F})=0$$ for $i>0$. Hence, the "\textit{topological Mittag-Leffler}" \cite[Remark~0.13.2.4]{EGA} condition implies that $R\Gamma(Y_S,\mathcal{F})=R\varprojlim R\Gamma(Y_{S,[q^{-i},q^i]},\mathcal{F})$. From which we conclude that $R\Gamma(Y_S,\mathcal{F})=\mathcal{F}(Y_S)[0]$.\\

Note now that since $Y_S$ is a $\mathbb{Z}$-torsor over $X_S$, we have that $\mathcal{E}\cong (p_*p^*\mathcal{E})^\mathbb{Z}$. Moreover this translates to an isomorphism of global sections,
\begin{align}\label{6}
    H^0(\mathcal{E})\xrightarrow{\cong}H^0((p_*p^*\mathcal{E})^\mathbb{Z})\cong H^0((p_*p^*\mathcal{E})))^\mathbb{Z}
\end{align}
Consider now the construction of $R\Gamma(X_S,\mathcal{E})$. We first find an injective resolution $\mathcal{E}\rightarrow\mathcal{I}^\bullet$, and then apply the global sections functor so that,
$$R\Gamma(X_S,\mathcal{E})=H^0(\mathcal{I}^1)\rightarrow H^0(\mathcal{I}^2)\rightarrow H^0(\mathcal{I}^3)\rightarrow\dots$$
We claim that $R\Gamma(X_S,\mathcal{E})=R\Gamma(\mathbb{Z},R\Gamma(Y_S,p^*\mathcal{E}))$, where   $R\Gamma(\mathbb{Z},(-))$ is the right derived functor of $\operatorname{Hom}_{E[t^{\pm 1}]}(E_{\mathrm{triv}},(-))$, and $E_\mathrm{triv}$ is the $E[t^{\pm 1}]$-module given the trivial action of $t$. This is the also the functor of $\mathbb{Z}$-invariants of $E$-vector spaces.\\

To this end, notice that $R\Gamma(Y_S,p^*\mathcal{E})$ is given by the complex $H^0(p_*p^*(\mathcal{I}^1))\rightarrow H^0(p_*p^*(\mathcal{I}^2))\rightarrow H^0(p_*p^*(\mathcal{I}^3))\rightarrow\dots$. We claim that each of the $H^0(p_*p^*(\mathcal{I}^1))$ are acyclic with respect to $R\Gamma(\mathbb{Z},(-))$. Indeed, 
\begin{align*}
    &\hspace{12pt}\operatorname{Ext}^i_{E[t^{\pm 1}]}(E_{\mathrm{triv}},H^0(p_*p^*(\mathcal{I})))\\
    &\cong \operatorname{Ext}^i_{E[t^{\pm 1}]}(E_{\mathrm{triv}},\prod_{i\in\mathbb{Z}}H^0(\mathcal{I}))\\
    &\cong\prod_{i\in\mathbb{Z}}\operatorname{Ext}^i_{E[t^{\pm 1}]}(E_{\mathrm{triv}},\bigoplus_n E_{\mathrm{triv}} )\\
    &\cong\bigoplus_n\prod_{i\in\mathbb{Z}}\operatorname{Ext}^i_{E[t^{\pm 1}]}(E_{\mathrm{triv}}, E_{\mathrm{triv}} )=0
\end{align*}
where the last equality follows from the fact a module $M$ over a PID $R$ is injective if and only if it is divisible.  \\

Together, this implies that $R\Gamma(\mathbb{Z},R\Gamma(Y_S,p^*\mathcal{E}))$ can be calculated by $H^0(p_*p^*(\mathcal{I}^1))^\mathbb{Z}\rightarrow H^0(p_*p^*(\mathcal{I}^2))^\mathbb{Z}\rightarrow H^0(p_*p^*(\mathcal{I}^3))^\mathbb{Z}\rightarrow\dots$, which along with the identity (\ref{6}) gives us the claim that $R\Gamma(X_S,\mathcal{E})=R\Gamma(\mathbb{Z},R\Gamma(Y_S,p^*\mathcal{E}))$.\\

Note now that we have a projective resolution of $E_{\mathrm{triv}}$ given by,
$$0\rightarrow E[t^{\pm 1}]\xrightarrow{t-1} E[t^{\pm 1}]\xrightarrow{t\rightarrow1} E_{\mathrm{triv}}\rightarrow 0.$$

Using the fact that $R\Gamma(Y_S,p^*\mathcal{E})=p^*\mathcal{E}(Y_S)[0]$, we can calculate $R\Gamma(\mathbb{Z},R\Gamma(Y_S,p^*\mathcal{E}))$ by
\begin{align*}
    \operatorname{Hom}_{E[t^{\pm 1}]}(E[t^{\pm 1}],H^0(Y_S,p^*\mathcal{E}))&\xrightarrow {(t-1)^*}\operatorname{Hom}_{E[t^{\pm 1}]}(E[t^{\pm 1}],H^0(Y_S,p^*\mathcal{E}))\\
    H^0(Y_S,p^*\mathcal{E})&\xrightarrow{\varphi-\mathrm{id}}H^0(Y_S,p^*\mathcal{E}).    
\end{align*}
In particular, since $R\Gamma(X_S,\mathcal{E})\cong [H^0(Y_S,p^*\mathcal{E})\xrightarrow{\varphi-\mathrm{id}}H^0(Y_S,p^*\mathcal{E})]$, we have that $H^i(X_S,\mathcal{E})=0$ for $i\geq2$.
\begin{subsubsection}{Vector bundles from isocrystals.}\label{5.2.1}
    We fix an algebraically closed field $k\hspace{1pt}|\mathbb{F}_q$, and denote by $\Breve{E}=W_{\mathcal{O}_E}(k)[\frac{1}{\pi}]$ the complete unramified extension of $E$ with residue field $k$. The Witt vector functor allows us to lift the Frobenius on $k$ to a Frobenius on $\Breve{E}$ which we call $\varphi_{\Breve{E}}$. If $S\in\mathrm{Perf}_k$, we have a map 
    $$Y_S\rightarrow \Spa(\Breve{E},\mathcal{O}_{\Breve{E}}).$$

    Note that any isomorphism of the trivial vector bundle $\mathcal{O}_{Y_S}^{\oplus n}$ that is twisted by $\varphi_S$ of the form
    \begin{align}\label{7}
        \xi:\mathcal{O}_{Y_S}^{\oplus n}\xrightarrow{\cong} (\varphi_{S})_*\mathcal{O}_{Y_S}^{\oplus n}
    \end{align} allows us to descend $\mathcal{O}_{Y_S}^{\oplus n}$ to $X_S$. \\

    We therefore consider the category of isocrystals, which are finite dimensional $\Breve{E}$-vector spaces with a $\varphi_{\Breve{E}}$ twisted isomorphism,
    $$\xi_{\Breve{E}}:\Breve{E}^{\oplus n}\xrightarrow{\cong}\varphi_{\Breve{E}}^*\Breve{E}^{\oplus n}$$
    The pairs $(\Breve{E}^{\oplus n},\xi_{\Breve{E}})$ are classified up to isomorphism as follows. We associate to a rational number $\lambda=\frac{s}{r}\in \mathbb{Q}$ such that $r>0$ and $(r,s)=1$, the $\Breve{E}$-vector space of dimension $r$ and the map defined by the matrix;

    \[\tilde\xi_\lambda=
 \begin{pmatrix}
    0&1\\
    &0&1\\
    &&&\ddots\\
    \pi^{-s}&&&&0&1
\end{pmatrix}\
 \in\mathrm{M}_{r\times r}(\Breve{E})\]   
precomposed with the diagonal Frobenius morphism. This defines a vector bundle on $\Spa(\Breve{E},\mathcal{O}_{\Breve{E}})$ along with a $\varphi_{\Spa(\Breve{E},\mathcal{O}_{\Breve{E}})}$ twisted isomorphism. Therefore, pulling this datum back to $Y_S$ gives exactly the data specified by (\ref{7}), allowing us to descend this to a vector bundle to $X_S$ which we call $\mathcal{O}(\lambda)$.

\begin{prop}[{\cite[Proposition~II.2.3]{fargues2024geometrizationlocallanglandscorrespondence}}]\label{5.26}
    Let $S^\sharp$ be an untilt over $E_\infty$ of a perfectoid space $S$. We have an exact sequence of $\mathcal{O}_{X_S}$-modules
    $$0\rightarrow\mathcal{O}_{X_S}\rightarrow\mathcal{O}_{X_S}(1)\rightarrow\mathcal{O}_{S^\sharp}\rightarrow0.$$
\end{prop}
    
\end{subsubsection}
\end{subsection}

\section{Banach-Colmez Spaces}
In this section we would like to introduce the notion of Banach-Colmez spaces and discuss their various properties. Following this, we introduce the Fourier transform on Banach-Colmez spaces discussed in a recent paper by Anschütz and Le Bras \cite{anslebras}.

\begin{defn}
    We define the \textit{absolute Banach-Colmez space} attached to an isocrystal as defined in section \ref{5.2.1}, as the v-sheaf on Perf prescribed by the following assignment

    $$BC(\mathcal{O}(\lambda)):S\in \text{Perf}\longmapsto R\Gamma(X_S,\mathcal{O}(\lambda)).$$
\end{defn}
\begin{notation}
    We denote by $\underline{E}$ the v-sheaf on Perf prescribed by 
    $$T\rightarrow \mathrm{Hom}_{cont'}(|T|,E)$$
\end{notation}
\begin{prop}
    Let $\lambda\in \mathbb{Q}$, and $\mathcal{O}(\lambda)$ be the vector bundle on $X_S$ situated in degree zero as described by section \ref{5.2.1}. 

    \begin{enumerate}
        \item If $\lambda<0$ then, $H^0(X_S,\mathcal{O}(\lambda))=0$.
        \item If $\lambda=0$ the map 
        $$\underline{E}\rightarrow\mathcal{BC}(\mathcal{O})$$ is an isomorphism of pro-étale sheaves.
        \item If If $\lambda>0$, then $H^1(X_S,\mathcal{O}(\lambda))=0$ for all affinoid $S\in$ Perf.
        \item If $0<\lambda<[E:\mathbb{Q}_p]$ (or in the equal characteristic case, for all $\lambda>0)$, we have an isomorphism 
        $$\mathcal{BC}(\mathcal{O}(\lambda))\cong\mathrm{Spd}\hspace{1pt}k \llbracket x_1^{1/p^\infty},\dots,x_r^{1/p^\infty}\rrbracket$$
    where $\lambda=\frac{s}{r}$ in reduced form.
    \end{enumerate}
\end{prop}
\begin{proof}
    We note that we can reduce to the case $\lambda=n\in \mathbb{Z}$, by replacing $E$ with its unique unramified extension of degree $r$.\\
    
    Note that we have a Čech cover of $X_S$ given by $Y_{S,[1/q,1]}\cup Y_{S,[1,q]}$ which gives us the following Čech complex,
    \begin{align*}
        0\rightarrow\mathcal{O}_{Y_S}({Y_{S,[1/q,1]}})\oplus\mathcal{O}_{Y_S}({Y_{S,[1,q]}})\xrightarrow{res^{[1/q,1]}_{[1,1]}-res^{[1,q]}_{[1,1]}}\mathcal{O}_{Y_S}({Y_{S,[1,1]}})
    \end{align*}
    Through the identification $\mathcal{O}_{Y_{S,[1/q,1]}}$ with $\mathcal{O}_{Y_{S,[1,q]}}$ using the descent isomorphism $\pi^{-s}\varphi$, this is quasi-isomorphic to
    \begin{align*}
        0\rightarrow\mathcal{O}_{Y_S}({Y_{S,[1,q]}})\xrightarrow{ res^{[1,q]}_{[1,1]}\circ(\pi^{-n}\varphi-\mathrm{id})}\mathcal{O}_{Y_S}({Y_{S,[1,1]}})\\
        0\rightarrow\mathcal{O}_{Y_S}({Y_{S,[1,q]}})\xrightarrow{ res^{[1,q]}_{[1,1]}\circ(\varphi-\pi^{n})}\mathcal{O}_{Y_S}({Y_{S,[1,1]}})
    \end{align*}
 We let $B_{R,[a,b]}:=\mathcal{O}_{Y_S}({Y_{S,[a,b]}})$, and note that we have a canonical injection $B_{R,[1,q]}\hookrightarrow B_{R,[1,1]}$. Which allows us to rewrite the above as
 $$0\rightarrow B_{R,[1,q]}\xrightarrow{ \varphi-\pi^{n}}B_{R,[1,1]}$$
 Therefore to see part (3) of the proposition, it suffices to show that this map is surjective. To this end, note that an element of $B_{R,[1,1]}$ can be written as the sum of an element in $B_{R,[0,1]}[\frac{1}{\pi}]$ and of an element in $[\varpi]B_{R,[1,\infty]}$. Now  one can check that $g:B_{R,[0,1]}[\frac{1}{\pi}]\rightarrow B_{R,[0,q]}$, given by $$g(f) =\varphi^{-1}(f)+\pi^n\varphi^{-2}(f)+\pi^{2n}\varphi^{-3}(f)+\dots$$ is a well defined map due to convergence, and furthermore that it acts as an inverse to $\varphi-\pi^{n}$. Similarly the function $g':[\varpi]B_{R,[1,\infty]}\rightarrow B_{R,[0,q]}$ given by $$g'(h)=-\pi^nh-\pi^{2n}\varphi(h)-\pi^{3n}\varphi^2(h)\dots$$ also acts as a well defined inverse. This suffices to show the claim that the map in consideration was surjective. \\

 Having proved part (3), we now move on to part (4). In the equal characteristic case, we have that $Y_S$ is isomorphic to the punctured adic open disc. In \cite[example~4.4]{hubner2024adic} it is shown that this implies that the global sections have the form, $$\Gamma(Y_S,\mathcal{O}_{Y_S})=\Big\{\sum_{i\geq0}^{\infty}a_i\pi^i|\hspace{3pt}|a_i|\cdot\rho^i\rightarrow0 \hspace{3pt}\text{for all}\hspace{2pt} \rho \in (0,1)\Big\}.$$ Through the descent isomorphism we have that $$H^0(X_S,\mathcal{O}_{X_S}(n))=H^0(Y_S,\mathcal{O}_{Y_S})^{\varphi=\pi^n}$$
 
  which in terms of the power series representation, becomes the requirement $\varphi\left(r_i\right)=r_{i+n}$. Therefore, we are allowed to freely choose the first $n$ coefficients of the power series, with the rest already being determined by such a choice. Furthermore, the decay condition implies that the choices must be topologically nilpotent, giving us the required isomorphism.\\

 The mixed characteristic case is proved using \cite[Theorem~A, Proposition~3.1.3(iii)]{moduli-p}, however we do not show it here.\\

 For part (2), recall that in proposition \ref{5.26} we introduced the exact sequence of $\mathcal{O}_{X_S}$-modules
 \begin{align}\label{8}
     0\rightarrow\mathcal{O}_{X_S}\rightarrow\mathcal{O}_{X_S}(1)\rightarrow\mathcal{O}_{S^\sharp}\rightarrow0
 \end{align}
 for which the associated long exact sequence simplifies by part (3) to the following
 $$0\rightarrow H^0(X_S,\mathcal{O}_{X_S})\rightarrow H^0(X_S,\mathcal{O}_{X_S}(1))\rightarrow H^0(X_S,\mathcal{O}_{S^\sharp})\rightarrow H^1(X_S,\mathcal{O}_{X_S})\rightarrow 0.$$
 Note that for $S=\Spa(R,R^+)$, we have that $H^0(X_S,\mathcal{O}_{S^\sharp})=R^\sharp$. Through  \cite[proposition~II.2.3]{fargues2024geometrizationlocallanglandscorrespondence}, we have that the map $H^0(X_S,\mathcal{O}_{X_S}(1))\rightarrow R^\sharp$, corresponds to the logarithm map from Lubin-Tate theory. In this situation we know that the kernel of this morphism is given by $\underline{E}$.\\

 For part (1), we twist the exact sequence in \ref{8}, to arrive at
 $$0\rightarrow\mathcal{O}_{X_S}(-1)\rightarrow\mathcal{O}_{X_S}\rightarrow\mathcal{O}_{S^\sharp}\rightarrow0$$
 For which the associated long exact sequence is,
 $$0\rightarrow H^0(X_S,\mathcal{O}_{X_S}(-1))\rightarrow H^0(X_S,\mathcal{O}_{X_S})\rightarrow H^0(X_S,\mathcal{O}_{S^\sharp})\rightarrow H^1(X_S,\mathcal{O}_{X_S}(-1))\rightarrow 0.$$
 which, by functoriality gives the following sequence of sheaves on  $\text{Perf}_S$,
 $$0\rightarrow\mathcal{BC}(\mathcal{O}(-1))|_S\rightarrow\underline{E}\rightarrow(\mathbb{A}^1_{S^\sharp})^{\diamondsuit}\rightarrow\mathcal{BC}(\mathcal{O}(-1)[1])|_S\rightarrow0 $$
 We now claim that the map $\underline{E}\rightarrow(\mathbb{A}^1_{S^\sharp})^{\diamondsuit}$ is injective. Explicitly, for $T \in \text{Perf}_S$ we have that the inclusion $T^\sharp\hookrightarrow X_T$ realizes $T^\sharp$ as a Cartier divisor on $X_T$, and therefore, the map is given by the restriction of a global section in $H^0(X_T,\mathcal{O}_{X_T})$ to a global section in $H^0(T^\sharp,\mathcal{O}_{T^\sharp})$.\\

Since it suffices to check that this map is injective on stalks, we are  reduced to checking the claim over points of $S$. These are parametrized by adic spectra of the form $\Spa(C,\mathcal{O}_C)$, where $C$ is non-archimedean field, with topology induced by the valuation ring $\mathcal{O}_C$. The topology on this space is a that of a totally ordered chain of specializations of points. Note however, that the topology on $E$ is totally disconnected, this implies that 
$$\underline{E}(\Spa(C,\mathcal{O}_C)=\mathrm{Hom}_{cont'}(|\Spa(C,\mathcal{O}_C)|,E)=E$$

proving the injectivity claim.

\end{proof}
\begin{subsection}{Fourier Transform on Banach-Colmez spaces}
    To build up the theory of the Fourier transform on the Fargues-Fontaine curve, we require a classification result for vector bundles on the curve over a complete algebraically closed non-archimedean field over $\mathbb{F}_q$. Using this we define the slope of a vector bundle, which helps us categorize the kinds of behaviors we expect to see.

    \begin{prop}[{\cite[Proposition~II.2.10]{fargues2024geometrizationlocallanglandscorrespondence}}] Let $C$ be a complete algebraically closed non-archimedean field over $\mathbb{F}_q$, and by abuse of notation we also denote by $C$ the adic spectrum $ \Spa(C,\mathcal{O}_C)$. We have an isomorphism 
    \begin{align*}
        \mathbb{Z}&\xrightarrow{\cong}\mathrm{Pic}(X_C)\\
        n&\mapsto\mathcal{O}_{X_C}(n)
    \end{align*}
    \end{prop}
    We call the inverse of this isomorphism the degree map. Thus, this proposition allows us to define the degree of a vector bundle $\mathcal{E}$ on $X_C$ as
    $$\mathrm{deg}(\mathcal{E})=\mathrm{deg}(\mathrm{det}(\mathcal{E}))\in \mathbb{Z}.$$
    Here, we denote by $\mathrm{det}(\mathcal{E})$ the determinant of the vector bundle $\mathcal{E}$, which is by construction a rank one vector bundle on $X_C$. Combing this with the obvious notion of the rank of a vector bundle $\mathrm{rk}(\mathcal{E})$, we define the slope $\mu(E)$ to be 
    $$\mu(\mathcal{E}):=\frac{\mathrm{deg}(\mathcal{E)}}{\mathrm{rk}(\mathcal{E)}}\in \mathbb{Q}$$
    \begin{defn}
        Let $\mathcal{E}$ be a vector bundle on $X_S$. We say that $\mathcal{E}$ has \textit{only positive} (resp. \textit{only negative} or \textit{only non-negative}) slopes if its pullback along any $X_{\Spa(C,C^+)}\rightarrow X_S$ has a positive slope (resp. negative or non-negative slope). Here, $\Spa(C,C^+)\rightarrow S$ is a geometric point of $S$. 
    \end{defn}
\end{subsection}
\begin{subsubsection}{Stacks in $\underline{E}$-vector spaces}
    \begin{defn}
        A Picard groupoid is a symmetric monoidal category in which all morphisms are invertible and such that the semigroup of isomorphism classes of objects forms a group. 
    \end{defn}
    
\begin{ex}
    The category whose objects are line bundles on a scheme, and morphisms are isomorphisms of line bundles form a Picard groupoid under the tensor product. Here the dual line bundle, acts as the inverse under the tensor product. 
\end{ex}
\begin{ex}\label{ex1}
    A commutative group can be viewed as a Picard groupoid, where the category has one object, and the morphisms correspond to group elements. Here the symmetric monoidal structure on the morphisms is given by group composition. 
\end{ex}
    More generally one can consider sheaves of Picard groupoids on a site $\mathcal{S}$. Through example \ref{ex1}, a sheaf of rings $\mathcal{A}$ on $\mathcal{S}$ can be viewed as a sheaf of Picard groupoids, which naturally has the structure of a commutative algebra object in the category. We can therefore make sense of $\mathcal{A}$-module objects in the category of Picard sheaves over $\mathcal{S}$.

\begin{prop}[{\cite[3.2]{anslebras}}]\label{6.9}
The category of $\mathcal{A}$-module objects in the category of sheaves of Picard groupoids on $\mathcal{S}$, is equivalent to the derived category of sheaves of $\mathcal{A}$-modules having non-zero cohomology sheaves only in degrees $-1,0$.
\end{prop}
For the special case of complexes $K$ of sheaves $\mathcal{A}$-modules such that that $K^i=0$ for $i\neq-1,0$, we can describe this functor explicitly. Given a complex $[K=K^{-1}\xrightarrow{d}K^0]$ the associated $\mathcal{A}$-module $\mathbb{V}(K)$  in the category of sheaves of Picard groupoids is the stack associated to the prestack $\mathbb{V}^{pre}(K)$ given by:
\begin{itemize}
    \item For $U\in \text{Perf}_S$, $\mathrm{Ob}(\mathbb{V}^{pre}(K)(U))=K^0(U)$.
    \item If $x,y\in K^0(U)$, a morphism from $x$ to $y$ is an element $f\in K^{-1}(U)$ such that $df=y-x$.
    \item The composition law on morphisms is given by the addition in $K^{-1}(U)$.
    \item The additive structure is given by the addition in $K^{-1}(U)$ and $K^0(U)$.
\end{itemize}

Furthermore given an $\mathcal{A}$-module $\mathcal{G}$ in the category of sheaves of Picard groupoids over $S$, we denote by $\mathcal{G}_{\bullet}$ the complex of sheaves under the equivalence identified in proposition \ref{6.9}.

\begin{defn}[{\cite[Definition~3.13]{anslebras}}] Let $S$ be a perfectoid space in Perf. We define a \textit{stack in $E$-vector spaces (over $S$)} to be an object in the category of $\underline{E}$-module objects in the category of v-sheaves of Picard groupoids on S.
    
\end{defn}
\begin{ex}
    We have the following identities,
    $$\mathbb{V}(\underline{E}[0])=\underline{E},\hspace{20pt}\mathbb{V}(\underline{E}[1])=[S/\underline{E}].$$
where in the second identification, we see $S$ as the terminal object in the category of sheaves on $\text{Perf}_S$.
\end{ex}

\begin{defn}[{\cite[Definition~3.14]{anslebras}}]
    Let $S\in$ Perf. Let $\mathcal{G}$ be a stack in $E$-vector spaces over $S$. We define the dual $\mathcal{G}^\vee$ to be the following
    $$\mathcal{G}^\vee:=\mathcal{H}om(\mathcal{G},[S/\underline{E}])$$
where the homomorphisms considered are as those of stacks of $E$-vector spaces. We say that $\mathcal{G}$ is \textit{dualizable} if the natural morphism $\mathcal{G}\rightarrow(\mathcal{G}^\vee)^\vee$ is an isomorphism.
\end{defn}
\begin{rmk}\label{6.13}
    By the considerations in proposition \ref{6.9}, for an $\underline{E}$-module $\mathcal{G}$ in the category of sheaves of Picard groupoids over $S$, we have the following
    $$(\mathcal{G}^\vee)_\bullet=\tau_{\leq0}R\mathcal{H}om(\mathcal{G}_\bullet,\underline{E}[1]).$$
\end{rmk}
The morphism of sites $\tau: (X_S)_v\rightarrow (S)_v$ introduced in proposition \ref{5.24}, allows us now to define the notion of a relative Banach-Colmez space associated with a vector bundle $\mathcal{E}$ on $X_S$. If $\mathcal{E}$ has only non-negative slopes, resp. only negative slopes, we define the v-sheaves on $\text{Perf}_S$ 

$$BC(\mathcal{E})=\tau_*\mathcal{E}\hspace{5pt}, \hspace{5pt}\text{resp.}\hspace{5pt}BC(\mathcal{E})=R^1\tau_*\mathcal{E}$$

\begin{prop}[{\cite[Corollary~3.10]{anslebras}}]\label{6.15}
Let $\mathcal{E}_1,\mathcal{E}_2$ be two vector bundles on $X_S$, we have a natural bijection 
$$R\mathcal{H}om_{\mathcal{O}_{X_S}}(\mathcal{E}_1,\mathcal{E}_2)=R\mathcal{H}om_{S_v}(R\tau_*(\mathcal{E}_1),R\tau_*(\mathcal{E}_2))$$

\end{prop}
\begin{prop}[{\cite[Proposition~II.3.4]{fargues2024geometrizationlocallanglandscorrespondence}}]\label{6.16}
Let $S \in$ Perf and $\mathcal{E}$ a vector bundle on $X_S$.
\begin{itemize}
    \item If $\mathcal{E}$ has only negative slopes, then $H^0(X_S,\mathcal{E})=0$.
    \item If $\mathcal{E}$ has only non-negative slopes, then there exists a pro-étale cover $\tilde{S}\rightarrow S$ such that $H^1(X_{\tilde{S}},\mathcal{E}|_{X_{\tilde{S}}})=0$.
    \item If $\mathcal{E}$ has only positive slopes, then there exists an étale cover $S'\rightarrow S$ such that for all afinoid $T\rightarrow S'$ we have that $H^1(X_{T},\mathcal{E}|_{X_{T}})=0$.
\end{itemize}
    
\end{prop}

\begin{prop}
    Let $\mathcal{E}$ be a vector bundle on $X_S$ having either only positive or only negative slopes. We have
    $$BC(\mathcal{E})^\vee=BC(\mathcal{E}^\vee)$$
\end{prop}
\begin{proof}

We assume that $\mathcal{E}$ has only positive slopes. Note that by the third part of proposition \ref{6.16}, we have that there exists an étale cover $S'\rightarrow S$ such that for all affinoid $T\rightarrow S'$ we have $H^1(X_T,\mathcal{E}|_{X_T})=0$. Note now that an  étale cover is also a v-cover, and moreover the sheaf $R^1\tau_*\mathcal{E}$ is given by the sheafification of $T\mapsto H^1(X_T,\mathcal{E}|_{X_T})$. Together, these imply that $R^1\tau_*\mathcal{E}=0$. In addition to this, the discussion in section \ref{5.2} gives us that $H^i(X_T,\mathcal{E}_{X_T})=0$ for $i\geq2$, for which the same argument  allows us to conclude the vanishing $R^i\tau_*\mathcal{E}=0$ for $i\geq 2$. All in all, this implies that $R\tau_*\mathcal{E}\cong \tau_*\mathcal{E}\cong BC(\mathcal{E})$, which gives us the following chain of identifications,
$$ R\mathrm{H}om_{T_v}(BC(\mathcal{E}),\underline{E}[1])\overset{\ref{6.15}}{\cong} R\mathrm{H}om_{\mathcal{O}_{X_T}}(\mathcal{E},\mathcal{O}_{X_T}[1])\cong R\Gamma(X_T,\mathcal{E}^\vee)[1] $$
Note now that $\mathcal{E}^\vee$ has only negative slopes, which means that by the first part of proposition \ref{6.16} we have that $R\Gamma(X_T,\mathcal{E}^\vee)\cong H^1(X_T,\mathcal{E}^\vee)[-1]$, and therefore $R\Gamma(X_T,\mathcal{E}^\vee)[1]\cong H^1(X_T,\mathcal{E}^\vee)$. Now once again, the sheafifcation arguments mentioned in the first part of the proof give us the required result.

\end{proof}

\begin{subsubsection}{The construction of the Fourier transform for Banach-Colmez spaces}
We fix a ring $\Lambda$ of characteristic prime to $p$. As in the case of the Fourier transform for the affine line, we will first develop some theory that allows us to relate characters $E\rightarrow\Lambda^\times$, to elements in a certain derived category associated to v-stacks. We start by introducing some definitions, that allows us to talk about sheaves on a v-stack, and furthermore about their derived categories. We do not work in the same generality of stacks of $E$-vector spaces as in \cite{anslebras}, but instead choose to focus on the case of Banach-Colmez spaces.

\begin{prop}The v-presheaf on Perf represented by a perfectoid space of characteristic $p$ is a sheaf. 
\end{prop}

\begin{defn}
    A v-sheaf $\mathcal{F}$ on Perf is called a small v-sheaf, if there exists a surjection of sheaves  $Y\rightarrow \mathcal{F}$ from a perfectoid space $Y$. 
\end{defn}
\begin{defn}

    A v-stack $X$ on Perf is called a small v-stack, if there exists a surjection of stacks from a perfectoid space $Y\rightarrow X$ for which $Y\times_XY$ is a small v-sheaf.
    
\end{defn}
 Let $X$ be a small v-stack and consider the site $X_v$ of all perfectoid spaces of characteristic $p$ over $X$, equipped with the v-topology. We have, associated to this site, the derived category of sheaves of $\Lambda$-modules, $D(X_v,\Lambda)$. \\

 We now recall what it means to be a quasi-compact, quasi separated (\textit{qcqs}) object in a categorical sense. An object in a topos is called quasi-compact if any covering family has a finite subcover. Similarly an object $Z$ is called quasi-separated if for any quasi-compact objects $X,Y$ over $Z$, the fiber product $X\times_ZY$ is quasi-compact.

 \begin{defn}
    Let $\mathcal{P}$ be a property of morphisms of perfectoid spaces. A morphism of v-sheaves, or v-stacks $f:X\rightarrow Y$ on Perf is said to have property $\mathcal{P}$, if for all morphisms $Z\rightarrow Y$ from a perfectoid space $Z$, the base change $Z\times_YX$ is a perfectoid space, and the base change map has property $\mathcal{P}$.
\end{defn}
 
 \begin{defn}
     A diamond $Y$ is called spatial if it is \textit{qcqs}, and the topological space $|Y|$ has a basis given by $|U_i|$, where $U_i$ are a quasi compact open subdiamonds of $Y$. In a more general fashion, $Y$ is called locally spatial, if it admits an open cover by spatial diamonds.
 \end{defn}

\begin{defn}
    A morphism of perfectoid spaces $X\rightarrow Y$ is called separated if $\Delta:X\rightarrow X\times_YX$ is a closed immersion.
\end{defn}
\begin{defn}
     A morphism $f:X\rightarrow Y$ of v-stacks is called locally separated if there exists an open over of $X$, over which $f$ becomes separated.
\end{defn}
Let $X$ be a small v-stack. If X is also a locally spatial diamond, the étale site $X_{ét}$ of $X$ is the category whose objects are maps $f: Y\rightarrow X$ from diamonds $Y$, such that $f$ is locally separated and étale. For the sake of completeness we give the definition of $D_{ét}(X,\Lambda)\subset D(X_v,\Lambda) $. Let $\widehat{D}(X_{ét},\Lambda)$ be the left completion of the derived category ${D}(X_{ét},\Lambda)$. 

\begin{defn}
    We define $D_{ét}(X,\Lambda)\subset D(X_v,\Lambda) $ to be the full subcategory consisting of all $A\in D(X_v,\Lambda)$ such that for all surjective maps $Y\rightarrow X$ from a locally spatial diamond $Y$, we have that $f^*A\in \widehat{D}(Y_{ét},\Lambda)$.
\end{defn}

\begin{rmk}
    Given a morphism $f:X\rightarrow Y$ of small v-stacks, we have the following functors
    \begin{enumerate}
    \item$f^*:D_{ét}(Y,\Lambda)\rightarrow D_{ét}(X,\Lambda)$
    \item ${R}f_*:D_{ét}(X,\Lambda)\rightarrow D_{ét}(Y,\Lambda)$
    \item $-\otimes_{\Lambda}^{\mathbb{L}}-: D_{ét}(Y, \Lambda) \times D_{ét}(Y, \Lambda) \rightarrow D_{ét}(Y, \Lambda)$
    \item $R \mathscr{H} \operatorname{om}_{\Lambda}(-,-): D_{ét}(Y, \Lambda)^{\mathrm{op}} \times D_{ét}(Y, \Lambda) \rightarrow D_{ét}(Y, \Lambda)$.
    \end{enumerate}
    Furthermore, in the case of the case of the Banach-Colmez space $BC(\mathcal{E})$ associated to a vector bundle $\mathcal{E}$ over a perfectoid space $S$, we have by \cite[Proposition~3.34]{anslebras}, that there exists a derived pushforward with proper support
   $${R}(\mathrm{pr})_!:D_{ét}(BC(\mathcal{E}),\Lambda)\rightarrow D_{ét}(S,\Lambda).$$
   
   These functors satisfy the usual adjoint identities that one would expect. Moreover, there also exists a derived pushforward for any base change of $pr:BC(\mathcal{E})\rightarrow S$ by another Banach-Colmez space.
\end{rmk}

Let $\mathrm{Rep}(G,\Lambda)$ denote the category of smooth representations of a locally pro-p group $G$ on $\Lambda$-modules. Moreover, let $\mathrm{Sh}_v([*/\underline{G}],\Lambda)$ denote the category of v-sheaves of $\Lambda$-modules on the v-stack $[*/\underline{G}]$. We construct a functor

$$\mathrm{Rep}(G,\Lambda)\rightarrow \mathrm{Sh}_v([*/\underline{G}],\Lambda)$$ 

by sending a smooth $G$-representation $V$ to the v-sheaf $\mathcal{F}_V$ of $\Lambda$-modules on $[*/\underline{G}]$, prescribed in the following manner. We know that the data of a perfectoid space $X$ over $[*/\underline{G}]$ is equivalent to a $\underline{G}$-torsor $\widetilde{X}\rightarrow X$. The value of $\mathcal{F}_V$ at $X\rightarrow [*/\underline{G}]$ determined by $\widetilde{X}\rightarrow X$, is given by the set of all continuous $G$-equivariant maps $|\widetilde{X}|\rightarrow V$.\\

Note that locally we can trivialize the torsor $\widetilde{X}$ so that it is isomorphic to $X\times\underline{G}$. Now one can check that $|X\times\underline{G}|=|X|\times G$, giving us an action of $G$ on $|\widetilde{X}|$.\\

By the theory of quotient stacks, if $X'$ is a v-cover of $X$, then the associated $\underline{G}$-torsor $\widetilde{X'}$ that parametrizes the morphism $X'\rightarrow[*/\underline{G}]$ is given by the pullback in the diagram below,
\[\begin{tikzcd}
	{\widetilde{X'}} & {\widetilde{X}} \\
	{X'} & X
	\arrow[from=1-1, to=1-2]
	\arrow[from=1-1, to=2-1]
	\arrow["{\underline{G}\mathrm{-torsor}}", from=1-2, to=2-2]
	\arrow["{\mathrm{v-cover}}"', from=2-1, to=2-2]
\end{tikzcd}\]

so that $\widetilde{X'}\rightarrow \widetilde{X}$ is also a v-cover. Since v-covers are in particular topological quotient maps, this ensures the exactness of the following sequence,
$$0\rightarrow\mathrm{Hom}_{\mathrm{G-eq}}(\widetilde{X},V)\rightarrow\mathrm{Hom}_{\mathrm{G-eq}}(\widetilde{X'},V)\rightrightarrows\mathrm{Hom}_{\mathrm{G-eq}}(\widetilde{X'}\times_{\widetilde{X}}\widetilde{X'},V)$$
which implies that $\mathcal{F}_V$ is indeed a v-sheaf.\\

We denote by $D(G,\Lambda)$ the derived category of smooth representations of $G$ on $\Lambda$-modules. The construction above is exactly the functor used in the following theorem, giving us an equivalence of categories.
\begin{prop}[{\cite[Theorem~V.1.1]{fargues2024geometrizationlocallanglandscorrespondence}}] We have a natural symmetric monoidal equivalence
$$D(G,\Lambda)\cong D_{ét}([*/\underline{G}],\Lambda).$$

\end{prop}

Note that we have a canonical map of v-stacks $[S/\underline{E}]\rightarrow[*/\underline{E}]$, which induces a pull back on the derived categories $f^*:D_{ét}([*/\underline{G}],\Lambda)\rightarrow D_{ét}([S/\underline{G}],\Lambda)$. By above proposition, for a character $\psi:E\rightarrow\Lambda^\times$, we have an associated element in $D_{ét}([*/\underline{G}],\Lambda)$, which we can pull back to get an element in $D_{ét}([S/\underline{G}],\Lambda)$ which we shall call $\mathcal{L}_\psi$.\\

We can now finally start to define the Fourier transform on Banach-Colmez spaces. Fixing a character $\psi:E\rightarrow\Lambda^\times$, consider the diagram 
\[\begin{tikzcd}
	& {BC(\mathcal{E})\times_SBC(\mathcal{E}^\vee)} && {[S/\underline{E}]} \\
	\\
	{BC(\mathcal{E}^\vee)} && {BC(\mathcal{E})}
	\arrow["\alpha"', from=1-2, to=1-4]
	\arrow["{\pi^\vee}", from=1-2, to=3-1]
	\arrow["\pi"', from=1-2, to=3-3]
\end{tikzcd}\]

\begin{defn}
    We define the Fourier transform associated to $\psi$ on the Banach-Colmez space $BC(\mathcal{E})$ to be the functor
    $$\mathcal{F}_\psi:D_{ét}(BC(\mathcal{E}),\Lambda)\rightarrow D_{ét}(BC(\mathcal{E^\vee}),\Lambda)$$
    given by 
    $$A\rightarrow\pi_{!}^{\vee}\left(\pi^*(A) \otimes \alpha^* \mathcal{L}_\psi\right).$$
\end{defn}
\begin{rmk}
    In \cite{anslebras}, they actually define the Fourier transform for a more general class of objects, which they call \textit{nice} stacks in $E$-vector spaces. The definitions of these objects are chosen such that they satisfy the most general available assumptions so that the derived pushforward is defined. \\
    
    In \cite[Theorem~4.2.2]{hataja}, they define what it means for a morphism $f$ to be \textit{"smooth-locally nice"}, and show that under these hypothesis the derived functors $f_!,f^!$ exists. An important property that is afforded by smooth locally nice morphisms is that they are stable under base change by morphisms which are representable in Artin v-stacks. This motivates the definition of a nice stack in $E$-vector spaces $f:\mathcal{G}\rightarrow S$ to be one such that $f$ and $f^\vee$ are representable in Artin v-stacks and are smooth locally nice. 
\end{rmk}

\begin{prop}[{\cite[Proposition~3.34]{anslebras}}]
    Let $BC(\mathcal{E})$ be the Banach-Colmez space associated to a vector bundle of either only positive or only negative slopes, then the Fourier transform $\mathcal{F}_\psi:D_{ét}(BC(\mathcal{E}),\Lambda)\rightarrow D_{ét}(BC(\mathcal{E^\vee}),\Lambda)$ gives an equivalence of categories (for any choice of nontrivial character $\psi$).
\end{prop}

\end{subsubsection}

\end{subsubsection}

\cleardoublepage
\phantomsection
\addcontentsline{toc}{section}{References}
\bibliographystyle{alphaurl}
\bibliography{ref.bib}

\end{document}